\documentclass[pdflatex,sn-mathphys-num]{sn-jnl}
\usepackage{graphicx}%
\usepackage{multirow}%
\usepackage{amsmath,amssymb,amsfonts}%
\usepackage{amsthm}%
\usepackage{mathrsfs}%
\usepackage[title]{appendix}%
\usepackage{xcolor}%
\usepackage{textcomp}%
\usepackage{manyfoot}%
\usepackage{booktabs}%
\usepackage{algorithm}%
\usepackage{algorithmicx}%
\usepackage{algpseudocode}%
\usepackage{listings}%
\usepackage[subrefformat=parens]{subcaption} 
\theoremstyle{thmstyleone}%
\newtheorem{theorem}{Theorem}
\newtheorem{lemma}{Lemma}
\newtheorem{proposition}{Proposition}%
\theoremstyle{thmstyletwo}%
\newtheorem{remark}{Remark}%
\theoremstyle{thmstylethree}%
\newtheorem{definition}{Definition}%
\newtheorem{assumption}{Assumption}
\newtheorem{corollary}{Corollary}
\begin{document}
\title[Proximal nonlinear conjugate gradient methods]{Proximal nonlinear conjugate gradient methods for minimizing composite functions}
\author*[1]{\fnm{Shodai} \sur{Hamana}}\email{sho\_1715@keio.jp}

\author[2]{\fnm{Yasushi} \sur{Narushima}}\email{narushima@keio.jp}

\affil*[1]{\orgdiv{School of Science for Open and Environmental Systems}, 
\orgname{Keio University}, 
\orgaddress{\city{Yokohama}, \country{Japan}}}

\affil[2]{\orgdiv{Department of Industrial and Systems Engineering}, 
\orgname{Keio University}, 
\orgaddress{\city{Yokohama}, \country{Japan}}}


\abstract{The nonlinear conjugate gradient methods are known to be an effective approach for standard unconstrained optimization problems especially for large-scale problems.
This paper proposes a proximal nonlinear conjugate gradient method, which extends the nonlinear conjugate gradient methods to composite objective functions, namely, the sum of a smooth nonconvex function and a nonsmooth convex function, and its extension to the case where the nonsmooth function is weakly convex.
The proposed method uses the \textit{forward-backward residual} which is defined by using the proximal mapping instead of the gradient and determines the search direction based on the three-term Hestenes-Stiefel (HS) formula. 
We establish the global convergence of the proposed algorithm under standard assumptions for both convex and weakly convex nonsmooth functions, and analyze its convergence rate in terms of stationarity.
In addition, when the smooth term is strongly convex, we prove that the generated sequence converges to the optimal solution and establish its convergence rate.
Finally, numerical experiments show that the proposed method is stable and achieves better performance than existing methods in both convex and nonconvex settings.}

\pacs[Acknowledgments]{This work was supported by Japan Society for the Promotion of Science (JSPS) KAKENHI Grant Numbers JP26K14720 and JP23K10999.
The authors would like to thank Dr. Shummin Nakayama for helpful comments on the numerical experiments.}
\maketitle

\section{Introduction}\label{sec:1}
In this paper, we consider the composite minimization problem
\begin{equation}\label{objective_function}
     \min_{{x} \in \mathbb{R}^n} \quad f({x}) := g({x}) + h({x}),
\end{equation}
where~$f : \mathbb{R}^n \to \mathbb{R} \cup \{+\infty\}$ is a composite function of~$g:\mathbb{R}^n\to\mathbb{R}$ and~$h:\mathbb{R}^n\to\mathbb{R} \cup \{+\infty\}$, where~$g$ is a continuously differentiable function with gradient $\nabla g$, and $h$ is a proper lower semi-continuous convex function.
Problem~\eqref{objective_function} plays an important role in various fields such as image processing~\cite{Chambolle2016} and machine learning~\cite{Mohammadi2023}.
Specifically, when the function~$h$ includes sparse regularization like $\ell_1$-norm, it is known that the optimal solution has sparsity. 
These characteristics have been widely applied such as LASSO~\cite{Tibshirani1996}, group LASSO~\cite{Yuan2006}, and $\ell_1$-regularized logistic regression~\cite{Shevade2003}.

For standard unconstrained optimization problems, namely~\eqref{objective_function} with $h({x})=0$, iterative methods based on the gradient of the objective function, such as the steepest descent method, nonlinear conjugate gradient methods, and quasi-Newton methods, are widely used.
On the other hand, to solve~\eqref{objective_function}, a typical approach is the proximal gradient method.
Recent advancements in proximal gradient methods include Nesterov-type acceleration~\cite{Beck2009}, nonmonotone variants~\cite{Li2015NIPS2015,Wright2009}, and extensions based on Bregman distances that relax smoothness assumptions~\cite{Bauschke2016,Teboulle2018} or address nonconvex settings~\cite{Li2015NIPS2015,Wang2024}.
Additionally, proximal quasi-Newton methods~\cite{Becker2019,Lee2014,Li2017MMOR,Nakayama2021,Nakayama2024,Narushima2023,Scheinberg2016} improve efficiency by utilizing second-order information, 
although each iteration becomes computationally more expensive, as they require computing weighted proximal mappings that are not prox-friendly.

In this paper, we focus on nonlinear conjugate gradient methods, which are highly effective for standard unconstrained optimization problems, namely~\eqref{objective_function} with $h({x})=0$. 
In each iteration, the required information is limited to the current point ${x}_k$, its gradient $\nabla g({x}_k)$, the previous direction ${d}_{k-1}$, and the gradient at the previous point $\nabla g({x}_{k-1})$.
Therefore, unlike quasi-Newton methods, they do not require storing a Hessian approximation of the objective function, leading to lower memory usage and improved computational efficiency.
Thus, the nonlinear conjugate gradient methods are effective methods, especially for large-scale optimization problems.
In recent years, nonlinear conjugate gradient methods that generate sufficient descent directions have been actively studied~\cite{Al-Baali2015,Hager2006,Hager2005,Kobayashi2017,nakamura2013,Narushima2012,Narushima2014,Zhang2006,Zhang2007,Zhang2006b}.
In particular, the three-term HS method proposed by Zhang et al.~\cite{Zhang2007} generates descent directions independently of the line search and is considered more efficient than other methods.
Despite these advantages, to the best of our knowledge, approaches based on the nonlinear conjugate gradient methods have not yet been thoroughly investigated in this context. 
Motivated by these observations, we propose a proximal nonlinear conjugate gradient method for solving~\eqref{objective_function}.
The main contributions of this paper are:
\begin{itemize}
    \item We propose a proximal nonlinear conjugate gradient method for minimizing a composite objective function by introducing the \textit{forward-backward residual} and employing a search direction based on the three-term HS formula~\cite{Zhang2007}.
    \item When $h({x}) = 0$ and $g({x})$ is a strongly convex function and appropriate parameter choices, the proposed method reduces to the nonlinear conjugate gradient methods. 
    This property implies our method is a natural extension of the nonlinear conjugate gradient methods.
    \item We prove the global convergence of the proposed method for both convex and weakly convex nonsmooth terms under standard assumptions.
    We establish the convergence rates of the proposed method under various structural settings. 
    For general nonconvex composite problems where the nonsmooth term is either convex or weakly convex, we establish an optimal global sublinear convergence rate of $\mathcal{O}(1/\sqrt{k})$ in terms of both the stationarity measure and the distance to the subdifferential, which matches the standard rates in the literature~\cite{Beck2017, Bohm2021, bookNesterov2004}. 
    Furthermore, when the smooth term $g$ is strongly convex, we establish that the distance to the optimal solution converges at an $\mathcal{O}(1/\sqrt{k})$ rate.
    \item Numerical comparisons with TFOCS~\cite{tfocs} and PNOPT~\cite{pnopt} demonstrate that our method achieves superior and stable performance across both convex and nonconvex settings.
\end{itemize}
The rest of this paper is organized as follows.
Section~\ref{sec:2} reviews the notation and existing methods.
Section~\ref{sec:3} details the proposed algorithm and establishes its global convergence.
Section~\ref{sec:weakly_convex} extends the proposed method to the weakly convex setting and establishes its convergence rate.
Section~\ref{sec:numerical_experiment} presents numerical experiments, and Section~\ref{sec:5} concludes the paper.

\section{Notations and preliminaries}\label{sec:2}
We provide some definitions of the mathematical concepts used in this paper and introduce some existing methods.
\subsection{Notations and definitions}\label{sec:2.1}
First, we explain some definitions of mathematical concepts relevant to convex analysis and optimization used in this paper. 
Let $\mathbb{R}$ represent the sets of real numbers.
For a vector ${x} \in \mathbb{R}^n$, the Euclidean norm on $\mathbb{R}^n$ will be denoted as $\|{x}\| = \sqrt{{x}^\top{x}}$, and also the $\ell_1$-norm will be denoted as $\|{x}\|_1 = \sum_{i=1}^{n}|x_i|$.
The distance from a point $x \in \mathbb{R}^n$ to a set $S \subseteq \mathbb{R}^n$ is denoted by $\text{dist}(x, S) := \inf_{y \in S} \Vert{}x - y\Vert{}$.
For a function $f$, we denote its domain by $\textnormal{dom}\:f:=\{{x} \in \mathbb{R}^n:f({x})< \infty \}$.
Given a convex function $h$ on $\mathbb{R}^n$, the subdifferential $\partial h({x})$ of $h$ at ${x}$ is 
\begin{equation*}
    \partial h({x}) = \{ {\xi} \in \mathbb{R}^n |\:h({u}) \geq h({x}) + {\xi}^\top({u}-{x}) ,\forall {u} \in \mathbb{R}^n \}.
\end{equation*}
We denote the set of stationary points associated with problem~\eqref{objective_function} as
\begin{equation}\label{definition:set_stationary_points}
\text{zer } \partial f =
\left\{
{x} \in \mathbb{R}^n | {0} \in \partial f({x})
\right\}
=
\left\{
{x} \in \mathbb{R}^n | {0} \in \nabla g({x}) + \partial h({x})
\right\}.
\end{equation}
If $g({x})$ is convex, then any stationary point is a global minimizer of~\eqref{objective_function}.
For given a constant $\mu > 0$ and a convex function $h:\mathbb{R}^n\to\mathbb{R}$, proximal mapping of $\mu h$ at a point ${v}\in \mathbb{R}^n$ is defined by
\begin{equation}\label{definition:proximal_operator_formula}
\operatorname{prox}_{\mu h}({v}) = \operatorname*{argmin}_{{x} \in \mathbb{R}^n} \left\{ h({x}) + \frac{1}{2\mu} \|{x} - {v}\|^2 \right\}.
\end{equation}
It was shown in~\cite{Beck2017} that the proximal mapping has several important properties.
\begin{proposition}\label{prop_prox_property}
\textnormal Given a constant $\mu > 0$ and a convex function $h:\mathbb{R}^n\to\mathbb{R}$, the following holds:
\begin{equation}\label{property:proximal_operetor_formula_1}
    \frac{1}{\mu} \left( {v}-\operatorname{prox}_{\mu h}({v}) \right)
    \in \partial h (\operatorname{prox}_{\mu h}({v})).
\end{equation}
\end{proposition}
\begin{proposition}\label{prop_prox_nonexpansive}
\textnormal Given a constant $\mu > 0$ and a convex function $h:\mathbb{R}^n\to\mathbb{R}$, the following holds:
\begin{equation}\label{property:proximal_operetor_formula_2}
    \|\operatorname{prox}_{\mu h}({u}) - \operatorname{prox}_{\mu h}({v}) \|
    \leq \| {u} - {v} \|.
    \end{equation}
\end{proposition}
\subsection{Nonlinear conjugate Gradient Methods} \label{sec:2.2}
In this section, we introduce nonlinear conjugate gradient methods, for solving:
\begin{equation*}\label{smooth_objective_function}
     \min_{{x} \in \mathbb{R}^n}  g({x}),
\end{equation*}
where the function $g:\mathbb{R}^n\to\mathbb{R}$ is a continuously differentiable. 
The prototype algorithm of nonlinear conjugate gradient methods is presented Algorithm~\ref{alg:nonlinear_cg_method}.
\noindent
\begin{algorithm}[htbp]
    \caption{Nonlinear conjugate gradient method} \label{alg:nonlinear_cg_method}
    \begin{algorithmic} 
        \State \textbf{Step 0:} Given an initial point ${x}_0 \in \mathbb{R}^n$. Set $k = 0$.
        \State \textbf{Step 1:} If the stopping condition is satisfied, terminate the algorithm.
        \State \textbf{Step 2:} Compute the search direction ${d}_k$ as:
            \begin{equation*}\label{nonlinear_cg:direction}
                {d}_k =
                \begin{cases}
                    -\nabla g({x}_0), & k=0, \\
                    -\nabla g({x}_k) + \beta_k {d}_{k-1}, & k \geq 1.
                \end{cases}
            \end{equation*}
        \State \textbf{Step 3:} Compute the step size $\alpha_k>0$ by a line search.
        \State \textbf{Step 4:} Update the next iterate by
            \begin{equation*}
                {x}_{k+1} = {x}_k + \alpha_k {d}_k.
            \end{equation*}
        \State \textbf{Step 5:} Set $k \gets k+1$ and return to Step 1.
    \end{algorithmic}
\end{algorithm}
\noindent
Usually, the algorithm is terminated when $\nabla g({x}_k)$ becomes sufficiently small.
Since numerical performance is significantly affected by the choice of $\beta_k$, various strategies for selecting $\beta_k$ have been extensively studied.
Well-known formulas for $\beta_k$ include those of the Fletcher-Reeves (FR) method, the Hestenes-Stiefel (HS) method, the Polak-Ribiere (PR) method, and the Dai-Yuan (DY) method~\cite{Hager2006}:
\begin{align*}
\beta_k^{FR} &= \frac{\|\nabla g ({x}_k)\|^2}{\|\nabla g ({x}_{k-1})\|^2}, \quad
\beta_k^{HS} = \frac{\nabla g ({x}_k)^\top{y}_{k-1}}{{d}_{k-1}^\top{y}_{k-1}}, \\
\beta_k^{PR} &= \frac{\nabla g ({x}_k)^\top {y}_{k-1}}{\|\nabla g ({x}_{k-1})\|^2}, \quad
\beta_k^{DY} = \frac{\|\nabla g ({x}_k)\|^2}{{d}_{k-1}^\top{y}_{k-1}},
\end{align*}
where ${y}_{k-1} = \nabla g ({x}_k) - \nabla g ({x}_{k-1})$. 
While the HS and PR methods are known to be numerically more efficient than other methods, they do not necessarily satisfy the following descent condition: for some constant $c > 0$,
\begin{equation}\label{descent_condition}
\nabla g({x}_k)^\top {d}_k \leq -c\|\nabla g({x}_k)\|^2
\end{equation}
holds for all $k>0$.
To overcome this weakness, improvements of the HS and PR methods have been actively studied~\cite{Al-Baali2015,Hager2006,Hager2005,Kobayashi2017,nakamura2013,Narushima2012,Narushima2014,Zhang2006,Zhang2006b,Zhang2007}.
For example, Zhang et al. proposed the three-term PR method~\cite{Zhang2006b} and the three-term HS method~\cite{Zhang2007}, which modify the search direction to ensure the descent condition~\eqref{descent_condition} without relying on the line search:
\begin{equation*}
{d}_k = -\nabla g({x}_k) + \beta_k^{PR}{d}_{k-1} - \frac{\nabla g({x}_k)^\top {d}_{k-1}}{\|\nabla g({x}_{k-1})\|^2}{y}_{k-1},
\end{equation*}
\begin{equation}\label{three_cg_HS_direction}
{d}_k = -\nabla g({x}_k) + \beta_k^{HS}{d}_{k-1} - \frac{\nabla g({x}_k)^\top {d}_{k-1}}{{d}_{k-1}^\top{y}_{k-1}}{y}_{k-1}.
\end{equation}
When an exact line search is performed, we have $\nabla g({x}_k)^\top {d}_{k-1} = 0$, and thus the method reduce to both the PR and HS formulas.
In this study, among such various search directions, we used a search method based on the three-term HS method~\eqref{three_cg_HS_direction} due to its straightforward convergence analysis and high computational efficiency.

\subsection{Proximal Gradient Method}\label{sec:2.3}
Next, we review proximal gradient methods, a representative technique for solving optimization problems such as~\eqref{objective_function}.
First, we impose the following assumption on the function $g$.
\begin{assumption}\label{assumption1}
The function \(g\) is a continuously differentiable function and its gradient \(\nabla g\) is Lipschitz continuous, namely, there exists a positive constant \(L\) such that
\begin{equation}\label{lipschitz_continuous_1}
\|\nabla g({u}) - \nabla g({v})\| \leq L \|{u} - {v}\|, \quad \forall {u}, {v} \in \mathbb{R}^n.
\end{equation}
\end{assumption}
\noindent
The proximal gradient methods solve the following subproblem (where $\mu_k>0$) to obtain the next iterate:
\begin{align}
{x}_{k+1}
&=\operatorname*{argmin}_{{x}\in\mathbb{R}^n}
\left\{g({x}_k)+{\nabla}g({x}_k)^\top({x}-{x}_k)
+\frac{1}{2\mu_k}\|{x}-{x}_k\|^2
+h({x})\right\}
\notag\\
&=\operatorname*{argmin}_{{x}\in\mathbb{R}^n}
\left\{h({x})
+\frac{1}{2\mu_k}
\|{x}-({x}_k-\mu_k\nabla g({x}_k))\|^2 \right\}.
\label{proximal_operator_3}
\end{align}
Using the definition of the proximal mapping~\eqref{definition:proximal_operator_formula},~\eqref{proximal_operator_3} can be transformed as:
\begin{equation}\label{proximal_operator_4}
{x}_{k+1}=\operatorname{prox}_{\mu_k h}({x}_k-\mu_k\nabla g({x}_k)).
\end{equation}
It implies that the proximal gradient method combines the steepest descent method with step size \(\mu_k\) and the proximal mapping. 
The parameter \(\mu_k\) must be chosen so that
\begin{equation}\label{PGM:decrease_condition}
g({x}_{k+1})
\leq g({x}_k)
+ \nabla g({x}_k)^\top ({x}_{k+1} - {x}_k)
+ \frac{1}{2\mu_k} \|{x}_{k+1} - {x}_k\|^2.
\end{equation}
The algorithm for the proximal gradient method is presented Algorithm~\ref{alg:PGM}.
\begin{algorithm}[htbp]
    \caption{Proximal gradient methods} \label{alg:PGM}
    \begin{algorithmic}
        \State \textbf{Step 0:} Given an initial point \({x}_0 \in \textnormal{dom}\:f\), a parameter \(\mu_{0} > 0\), and \(\theta \in (0, 1)\). Set \(k = 0\).
        \State \textbf{Step 1:} Compute:
        \[
        {x}_k^+ = \operatorname{prox}_{\mu_k h}({x}_k - \mu_k \nabla g({x}_k)).
        \]
        \State \textbf{Step 2:} If the following inequality holds, go to Step 4. Otherwise, set \(\mu_k \leftarrow \theta \mu_k\) and return to Step 1.
        \begin{equation}\label{PGM:decreace_condition_2}
            g({x}_k^+) 
            \leq g({x}_k) 
            + \nabla g({x}_k)^\top ({x}_k^+ - {x}_k) 
            + \frac{1}{2\mu_k} \|{x}_k^+ - {x}_k\|^2.
        \end{equation}
        \State \textbf{Step 3:} Update \(\mu_{k+1} \gets \mu_k\) for the parameter \(\mu_k\) that satisfies equation~\eqref{PGM:decreace_condition_2}.
        \State \textbf{Step 4:} Update the next iteration as:
        \[
        {x}_{k+1} = {x}_k^+.
        \]
        \State \textbf{Step 5:} If the stopping condition is satisfied, terminate the algorithm.
        \State \textbf{Step 6:} Set \(k \gets k+1\) and return to Step 1.
    \end{algorithmic}
\end{algorithm}\noindent
The set of stationary points $\text{zer } \partial f$ of the optimization problem~\eqref{objective_function} coincides with the set of fixed points of the iteration~\eqref{proximal_operator_4} with $\mu_k=\mu$.
This correspondence is formalized in the following proposition (see~\cite{Beck2017}):
\begin{proposition}\label{proposition_zer}
    Let $f=g+h$, where $g$ is a continuously differentiable function and $h$ is a proper lower semi-continuous convex function. Then, for any constant $\mu>0$, the following holds:
    \begin{equation*}
    {x} = \operatorname{prox}_{\mu h}({x}-\mu\nabla g({x})) 
    \Longleftrightarrow {x} \in \text{zer } \partial f.
\end{equation*}
\end{proposition}
\noindent
Therefore, the algorithm is terminated when \( \| {x}_{k+1} - {x}_k \| \) becomes sufficiently small.  
Under Assumption~\ref{assumption1}, the sequence \( \{ {x}_k \} \) generated by Algorithm~\ref{alg:PGM} satisfies the following property (see~\cite{Beck2017}):
\begin{equation}\label{theorem:PGM_decrease_condition_formula}
f({x}_{k+1})
\leq f({x}_k) - \frac{1}{2 \mu_k} \| {x}_{k+1}-{x}_k\|^2.
\end{equation}
\noindent
\section{Proposed algorithm and its global convergence}\label{sec:3}
In this section, we introduce the proximal nonlinear conjugate gradient methods for solving~\eqref{objective_function}.
We also verify that, when the objective function satisfies $h({x})=0$ and $g({x})$ is  a strongly convex function under appropriate parameter choices, the proposed methods reduce to the nonlinear conjugate gradient methods. Finally, we show the global convergence and $R$-linear convergence properties.
\subsection{Proximal nonlinear conjugate gradient methods}\label{sec:3.1}
In this section, we propose a proximal nonlinear conjugate gradient method for solving~\eqref{objective_function}.
For this purpose, we introduce the function \( {\eta}_\mu({x}) \) as follows:
\begin{equation}\label{function_eta}
            {\eta}_\mu({x})=-\frac{1}{\mu}({x}^+ - {x}),
\end{equation}
where ${x}^+={\rm prox}_{\mu h}({x}-\mu\nabla g({x}))$ and $\mu$ is a positive constant.
This function is referred to as the \textit{forward-backward residual} in~\cite{Themelis2018,Themelis2019}.
Note that when $h({x})=0$, the relation ${\eta}_\mu({x}) = \nabla g({x})$ holds.
Furthermore, Proposition~\ref{proposition_zer} implies the following equivalence:
\begin{equation}\label{stationary_point_condition}
    {\eta}_\mu({x})=0 \Longleftrightarrow {x}\in\text{zer } \partial f.
\end{equation}
Next, we establish a fundamental relationship among the \textit{forward-backward residual} ${{\eta}_\mu}({x})$ and  the gradient of $g$, and the subgradient of $h$.
\begin{lemma}\label{lem_mu_optimality}
The following relationship holds:
\begin{equation}\label{lem_mu_optimality_formula}
   {\eta}_\mu({x}) - \nabla g ({x}) \in \partial h ({x}^+).
\end{equation}
\end{lemma}
\begin{proof}
Letting ${v}={x}-\mu \nabla g ({x})$ in~\eqref{property:proximal_operetor_formula_1}, we obtain:
\[
\frac{1}{\mu}\left({x} - \mu\nabla g ({x}) - {x}^+\right) \in \partial h ({x}^+).
\]
Thus, we have
\[
\frac{1}{\mu}({x}-{x}^+) 
-\nabla g ({x}) 
\in \partial h ({x}^+),
\]
which implies~\eqref{function_eta}.
\end{proof}
Our method extends the nonlinear conjugate gradient framework to~\eqref{objective_function} by using the \textit{forward-backward residual} ${{\eta}_\mu({x})}$ instead of the gradient $\nabla g({x})$. 
In particular, at each iteration $k$, we use
\begin{equation}\label{eta_k}
            {\eta}_k \equiv {\eta}_{{\mu}_k}({x}_k)=
            -\frac{1}{\mu_k}({x}_k^+ - {x}_k)
\end{equation}
where ${x}_k^+={\rm prox}_{\mu_k h}({x}_k-\mu_k\nabla g({x}_k))$.
The parameter $\mu_k$ is chosen to satisfy the following inequality:
\begin{equation}\label{proposed_decrease_condition}
    g({x}_{k}^+)
    \leq g({x}_k)
    + \nabla g({x}_k)^\top ({x}_{k}^+ - {x}_k)
    + \frac{1}{2\mu_k} \|{x}_{k}^+ - {x}_k\|^2.
\end{equation}
This inequality is the same as~\eqref{PGM:decrease_condition} in the proximal gradient method.
Moreover, the condition is always satisfied whenever $\frac{1}{\mu_k}>L$.
The iterate is updated by:
\begin{equation*}
    {x}_{k+1} = {x}_k + \alpha_k{d}_k,
\end{equation*}
where $\alpha_k > 0$ is the step size and ${d}_k$ is the search direction.
For the search direction, we focus on the three-term HS method~\eqref{three_cg_HS_direction}, which is one of the most effective nonlinear conjugate gradient methods introduced in Section~\ref{sec:2.2}.
By using $\eta_{k}$, the search direction ${d}_k$ is given by
\begin{equation}\label{d_k}
    {d}_k=
    \begin{cases}
        -{\eta}_k,&k=0,\\
        -{\eta}_k + \beta_k {d}_{k-1} - \gamma_k {y}_{k-1},&k\geq1 ,
    \end{cases}
\end{equation}
where the parameters $\beta_k$ and $\gamma_k$ are defined by
\begin{equation}\label{beta_gamma}
    \beta_k = \frac{{\eta}_k^\top{y}_{k-1}}{{d}_{k-1}^\top{z}_{k-1}},
    \quad
    \gamma_k = \frac{{\eta}_k^\top{d}_{k-1}}{{d}_{k-1}^\top{z}_{k-1}},
\end{equation}
and ${y}_{k-1}$ is
\begin{equation}\label{y_k-1}
    {y}_{k-1} = {\eta}_k - {\eta}_{k-1}.
\end{equation}
To ensure global convergence, we define the corrected vector ${z}_{k-1}$ by modifying ${y}_{k-1}$ as follows:
\begin{equation*}
    {z}_{k-1} = {y}_{k-1} + \nu_k {s}_{k-1},
\end{equation*}
where $\nu_k$ is a non-negative parameter, and ${s}_{k-1}$ is defined by 
\begin{equation}\label{s_k-1}
    {s}_{k-1} = {x}_k -{x}_{k-1} (= \alpha_{k-1} {d}_{k-1}).
\end{equation}
Furthermore, the parameter ${\nu}_k$ is chosen such that
\begin{equation}\label{condition_s_k}
    {s}_{k-1}^\top {z}_{k-1} \geq \bar{\nu} \| {s}_{k-1} \|^2
\end{equation}
holds for some positive constant $\bar{\nu}$.
For example, for a positive constant $\hat{\nu}$, we can set 
\begin{equation}\label{condition_s_k_example}
    \nu_k =
    \begin{cases} 
        0, & \text{if } {s}_{k-1}^\top {y}_{k-1} \geq \hat{\nu} \|{s}_{k-1}\|^2, \\
        \max \left\{ 0, -\frac{{s}_{k-1}^\top {y}_{k-1}}{{s}_{k-1}^\top {s}_{k-1}} \right\}+ \hat{\nu} , & \text{otherwise}.
    \end{cases}
\end{equation}
If ${s}_{k-1}^\top {y}_{k-1} \geq \hat{\nu} \|{s}_{k-1}\|^2$,~\eqref{condition_s_k} holds with $\bar{\nu}=\hat{\nu}$. 
On the other hand, if ${s}_{k-1}^\top {y}_{k-1} < \hat{\nu} \|{s}_{k-1}\|^2$, then we have
\[
{s}_{k-1}^\top {z}_{k-1} = \max \{ {s}_{k-1}^\top {y}_{k-1} , 0 \} 
+ \hat{\nu} \|{s}_{k-1}\|^2
\geq \hat{\nu} \|{s}_{k-1}\|^2
\]
which implies~\eqref{condition_s_k} with $\bar{\nu}=\hat{\nu}$.
To determine the step size $\alpha_k$, we employ a two-stage line search strategy.
First, we determine the trial step size $t$ by finding the largest value from the sequence $\{1, \theta, \theta^2, \dots \}$ with $\theta \in (0, 1)$ that satisfies the following condition:
\begin{equation}\label{conditional_branch}
t\nabla g({x}_k)^\top {d}_k + h({x}_k+t{d}_k) - h({x}_k)
\leq -tT\|{\eta}_k\|^2,
\end{equation}
where $T > \delta$ and $\delta\in (0,1)$. If the trial step size satisfies $t > \bar{t}$ for a predefined threshold $\bar{t} \in (0, 1)$, we proceed to the second stage.
In this stage, we set $\alpha_k = t$ as the initial candidate and perform backtracking by iteratively updating $\alpha_k \leftarrow \tau \alpha_k$, where $\tau \in (0, 1)$, until the Armijo condition is satisfied:
\begin{equation}\label{linesearch_condition_for_proposed}
    f({x}_k + \alpha_k {d}_k) 
    \leq f({x}_k) + \delta\alpha_k{\eta}_k^\top{d}_k .
\end{equation}
Once $\alpha_k$ is determined, we update the iterate by ${x}_{k+1} = {x}_k + \alpha_k {d}_k$.
If no $t > \bar{t}$ satisfies condition~\eqref{conditional_branch}, we instead set $x_{k+1} = x_k^+$, thereby switching to the proximal gradient method.
Based on these arguments, we propose the Algorithm~\ref{alg:proposed}.
\begin{algorithm}[htbp]
    \caption{Proximal nonlinear conjugate gradient method} \label{alg:proposed}
    \begin{algorithmic}
        \State \textbf{Step 0:} 
        Given an initial point ${x}_0 \in \textnormal{dom}\:f$, and parameters $\mu_{-1}>0$, $\kappa\in(0,1)$, $\bar{\nu}>0$, $\delta\in(0,1)$, $T>\delta$, $\theta\in(0,1)$, $\bar t\in(0,1)$, $\tau\in(0,1)$. Set $k=0$.
        \State \textbf{Step 1:} 
        Compute ${x}_k^+$ by
        \begin{equation*}
            {x}_k^+={\rm prox}_{\mu_k h}({x}_k-\mu_k\nabla g({x}_k)),
        \end{equation*}
        where $\mu_k$ is the largest value among $\{\mu_{k-1},\mu_{k-1}\kappa,\mu_{k-1}\kappa^2,\cdots\}$ that satisfies~\eqref{proposed_decrease_condition}.
        \State \textbf{Step 2:} If the stopping condition is satisfied, terminate the algorithm.
        \State \textbf{Step 3:} Compute the search direction ${d}_k$ by (\ref{d_k}).
        \State \textbf{Step 4:} Find the largest trial step size $t \in \{1, \theta, \theta^2, \dots \}$ that satisfies condition~\eqref{conditional_branch}.
        If such a $t$ exists and satisfies $t > \bar{t}$, set $\alpha_{\text{init}} = t$ and go to Step 4.1; otherwise, go to Step 4.2.
        \State \textbf{\; Step 4.1:} Let $\alpha_k$ be the largest value among $\{\alpha_{\text{init}}, \tau \alpha_{\text{init}}, \tau^2 \alpha_{\text{init}}, \dots \}$ that satisfies the Armijo condition~\eqref{linesearch_condition_for_proposed}. Update the iterate by:
        \begin{equation*}
            {x}_{k+1} = {x}_k + \alpha_k {d}_k,
        \end{equation*}
        and go to Step~5.
        \State \textbf{\; Step 4.2:} Set ${d}_k=-{\eta}_k$, and update the iterate by:
        \[
        {x}_{k+1} = {x}_k^+ ,
        \]
        and go to Step~5.
        \State \textbf{Step 5:} Set $k \leftarrow k+1$ and return to Step 1.
    \end{algorithmic}
\end{algorithm}
\noindent
We terminate the Algorithm~\ref{alg:proposed} if $\|{x}_k^+ -{x}_k\|$ is sufficiently small.
This is justified by the fact that $x_k$ is a stationary point of~\eqref{objective_function} if and only if $\eta_k = 0$, which is equivalent to $x_k^+ - x_k = 0$ according to \eqref{stationary_point_condition}.
From the fact that $\mu_k$ is non-decreasing and from~\eqref{proposed_decrease_condition}, it follows that $\mu_k$ becomes constant for sufficiently large $k$.
\begin{remark}
If the Lipschitz constant $L_g$ of the gradient $\nabla g$ is known prior to computation, the search for $\mu_k$ in Step~1 is not required.
Specifically, by setting a fixed step size $\mu_k = \bar{\mu} \in (0, 1/L_g]$ for all $k$, the decrease condition~\eqref{proposed_decrease_condition} is automatically satisfied without backtracking line search.
\end{remark}

Next, we show that when the objective function is a strongly convex function under appropriate parameter choices, the proposed method reduce to the three-term HS method~\eqref{nonlinear_cg:direction}. 
First, we make the following assumption for the objective function.
\begin{assumption}\label{assumption_strong_convex}
    The function $g$ is a strongly convex. 
\end{assumption}
In other words, the following inequality holds for some positive constant $m > 0$:
\begin{equation}\label{formula_strongly_convex}
(\nabla g({x}) - \nabla g({y}))^\top ({x} - {y}) \geq m \|{x} - {y}\|^2.
\end{equation}
Specifically, we consider the case where \( g({x}) \) is a strongly convex function and \( h({x}) = 0 \) in problem~\eqref{objective_function}, which implies that ${\eta}_k = \nabla g ({x}_k)$.
Furthermore, the inequality in condition~\eqref{condition_s_k_example} becomes
\begin{equation*}
    ({x}_k - {x}_{k-1})^\top (\nabla g ({x}_k) - \nabla g ({x}_{k-1})) \geq \bar{\nu} \|{x}_k - {x}_{k-1}\|^2.
\end{equation*}
Since the equation~\eqref{formula_strongly_convex}, we have:
\begin{equation*}
    ({x}_k - {x}_{k-1})^\top (\nabla g ({x}_k) - \nabla g ({x}_{k-1})) \geq m \|{x}_k - {x}_{k-1}\|^2.
\end{equation*}
When $\bar{\nu} \leq m$, the condition~\eqref{condition_s_k_example} is always satisfied with \( \nu_k = 0 \).
Therefore, we obtain ${z}_{k-1} = {y}_{k-1}$. 
In this case, the parameters \( \beta_k \) and \( \gamma_k \) in Algorithm~\ref{alg:proposed} can be rewritten as follows:
\begin{equation*}
    \beta_k = \frac{\nabla g ({x}_k)^\top {y}_{k-1}}{{d}_{k-1} ^\top {y}_{k-1}},\quad
    \gamma_k = \frac{\nabla g ({x}_k)^\top {d}_{k-1}}{{d}_{k-1} ^\top  {y}_{k-1}}.
\end{equation*}
Therefore, the search direction coincides with that of the three-term HS method.
Moreover, the condition~\eqref{conditional_branch} can be rewritten as  
\[
\nabla g ({x}_k)^\top {d}_{k} \leq -T \| \nabla g ({x}_k) \|^2.
\]
Given the descent property $\nabla g(x_k)^\top d_k = -\|\nabla g(x_k)\|^2$, this condition is inherently satisfied for any $T \leq 1$.
Consequently, when the objective function is a strongly convex function and the parameters \( \bar{\nu} \) and \( T \) are selected that $\bar{\nu} \leq m$ and \( T \leq 1 \), then Algorithm~\ref{alg:proposed} becomes equivalent to the three-term HS method.
\subsection{Global convergence}\label{sec:3.2}
In this section, we establish the global convergence of the sequence $\{x_k\}$ generated by Algorithm~\ref{alg:proposed}.
We first derive several key inequalities used in the analysis.
From~\eqref{lipschitz_continuous_1} that for all \({u}, {v} \in \mathbb{R}^n\)
\begin{equation}\label{lipschitz_continuous_2}
g({u}) \leq g({v}) + \nabla g({v})^\top ({u} - {v}) + \frac{L}{2} \|{u} - {v}\|^2.
\end{equation}
As mentioned above, $\mu_k$ becomes constant for sufficiently large $k$.
Hence, without loss of generality, we assume that $\mu_k = \mu > 0$ for all $k \ge 0$.
The following lemma corresponds the descent condition~\eqref{descent_condition}.
\begin{lemma}\label{lem_descent_condition}
The search direction ${d}_k$ given by~\eqref{d_k} satisfies the following condition:
\begin{equation*}
    {\eta}_k^\top{d}_k = -\|{\eta}_k\|^2.
\end{equation*}
\end{lemma}
\begin{proof}
Since the case for $k = 0$ is trivial, we focus on the case where $k \geq 1$.
It follows from~\eqref{beta_gamma}, we have
\begin{align*}
    {\eta}_k^\top{d}_k &= {\eta}_k^\top(-{\eta}_k + \beta_k {d}_{k-1} - \gamma_k{y}_{k-1}) \\
   &= -\|{\eta}_k\|^2 
   + \frac{{\eta}_k^\top{y}_{k-1}}{{d}_{k-1}^\top{z}_{k-1}}
   {\eta}_k^\top{d}_{k-1}
    - \frac{{\eta}_k^\top{d}_{k-1}}{{d}_{k-1}^\top{z}_{k-1}}
    {\eta}_k^\top{y}_{k-1} \\
   &= -\|{\eta}_k\|^2.
\end{align*} 
\end{proof}
Using this lemma, the Armijo condition~\eqref{linesearch_condition_for_proposed} can be rewritten by
\begin{equation}\label{linesearch_condition_for_proposed_2}
    f({x}_k + \alpha_k {d}_k) 
    \leq f({x}_k) - \delta\alpha_k \|{\eta}_k\|^2.
\end{equation}
Next, we provide an evaluation of the $\|{y}_{k-1}\|$.
\begin{lemma}\label{lem_yk}
Suppose Assumption \textnormal{\ref{assumption1}} is satisfied.
Then, there exists a positive constant $K$  such that the following inequality holds:
\begin{equation} \label{lem_yk_formula}
       \|{y}_{k-1}\| \leq K \| {s}_{k-1} \|.
\end{equation}
\end{lemma}
\begin{proof}
By \eqref{property:proximal_operetor_formula_2}, \eqref{lipschitz_continuous_1}, \eqref{eta_k} and \eqref{y_k-1}, we have
\begin{align*}
\|y_{k-1}\|
&= \|\eta_k - \eta_{k-1}\| \\
&\leq \frac{1}{\mu}\|{x}_k - {x}_{k-1}\| 
    + \frac{1}{\mu}\|\operatorname{prox}_{\mu h}({x}_k - \mu\nabla g({x}_k)) 
    - \operatorname{prox}_{\mu h}({x}_{k-1} - \mu\nabla g({x}_{k-1}))\| \\
    &\leq \frac{1}{\mu}\|{x}_k - {x}_{k-1}\|
    + \frac{1}{\mu}\|({x}_k - \mu\nabla g({x}_k)) 
    - ({x}_{k-1} - \mu\nabla g({x}_{k-1}))\| \\
    &\leq \frac{2}{\mu}\|{x}_k - {x}_{k-1}\| + \|\nabla g({x}_k)- \nabla g({x}_{k-1})\|\\
    &\leq \left(\frac{2}{\mu}+L\right)\| {x}_k - {x}_{k-1} \|.
\end{align*}
This proves~\eqref{lem_yk_formula} with $K=\frac{2}{\mu}+L$.
\end{proof}
Next, we provide an evaluation of the $\|{d}_{k}\|$.
\begin{lemma}\label{lem_dk}
Suppose Assumption~\ref{assumption1} is satisfied.
Then, there exists a positive constant $P$ such that
\begin{equation}\label{lem_dk_formula}
    \|{d}_k\| \leq P\|{\eta}_k\|.
\end{equation}
\end{lemma}
\begin{proof}
The case $k=0$ is trivial. 
For $k \ge 1$, it follows from~\eqref{d_k} that
\begin{align*}\label{d_k_1}
    \|{d}_k\| 
    &\leq \| {\eta}_k \| + \|\beta_k{d}_{k-1}\| + \| \gamma_k{y}_{k-1} \| \notag\\
    &\leq \| {\eta}_k \| 
    + \frac{\|{\eta}_k\| \|{y}_{k-1}\|}
    {\left|{d}_{k-1}^\top {z}_{k-1}\right|}  \| {d}_{k-1} \| 
    + \frac{\|{\eta}_k\| \|{d}_{k-1}\|}{\left|{d}_{k-1}^\top {z}_{k-1}\right|}  \|{y}_{k-1} \|\notag \\
    &= \| {\eta}_k \| 
    + 2 \, \frac{\|{\eta}_k\| \|{y}_{k-1}\| \| {d}_{k-1} \| }
    {\left|{d}_{k-1}^\top {z}_{k-1}\right|}.
\end{align*}
It follows from~\eqref{s_k-1} and \eqref{condition_s_k} that
\begin{align*}
    {d}_{k-1}^\top {z}_{k-1} 
    = \frac{1}{\alpha_{k-1}} ({x}_k - {x}_{k-1})^\top {z}_{k-1} 
    = \frac{1}{\alpha_{k-1}} {s}_k^\top {z}_{k-1} 
    &\geq \frac{\bar \nu}{\alpha_{k-1}} \|{s}_{k-1}\|^2.
\end{align*}
Therefore, we have
\begin{align*}
    \|{d}_k\|
    &\leq \| {\eta}_k \| 
    + 2 \, \frac{\|{\eta}_k\| \|{y}_{k-1}\| \| {d}_{k-1} \| }
    {\left|{d}_{k-1}^\top {z}_{k-1}\right|} \\
    &\leq \| {\eta}_k \| 
    + \frac{2}{\alpha_{k-1}}  \frac{\|{\eta}_k\| \|{y}_{k-1}\| \| {s}_{k-1} \| }
    {\frac{\bar{\nu}}{\alpha_{k-1}}\|{s}_{k-1}\|^2}\\
    &\leq \left(1+\frac{2K}{\bar{\nu}}\right) \| {\eta}_k \|,
\end{align*}
where the last inequality follows from~\eqref{lem_yk_formula}.
This proves~\eqref{lem_dk_formula} with $P=1+\frac{2K}{\bar{\nu}}$.
\end{proof}
The following lemma is useful to guarantee the well-definedness of the line search.
\begin{lemma}\label{lem_g_lipchitz}
Suppose Assumption~\ref{assumption1} is satisfied, and let $\alpha \in (0,1]$. 
For any ${x}_k,{d}_k \in \mathbb{R}^n$, the following holds:
\begin{equation}\label{lem_g_lipchitz_formula}
    f({x}_k+\alpha {d}_k) 
    \leq f({x}_k) 
    +\alpha \nabla g({x}_k)^\top{d}_k
    + h({x}_k+\alpha {d}_k) 
    - h({x}_k)
    +\frac{L}{2}\|\alpha {d}_k\|^2
\end{equation}
\end{lemma}
\begin{proof}
It follows from~\eqref{lipschitz_continuous_2} that
\begin{align*}
    f({x}_k+\alpha {d}_k) 
    &= g({x}_k+\alpha {d}_k) + h({x}_k+\alpha {d}_k) \\
    &\leq g({x}_k) +\alpha \nabla g({x}_k)^\top{d}_k
    +\frac{L}{2}\|\alpha {d}_k\|^2 + h({x}_k+\alpha {d}_k).
\end{align*}
\end{proof}
The following lemma guarantees that if there exists a trial step size $t$ satisfying condition~\eqref{conditional_branch} during the line search, then the same condition is preserved for any smaller step size.
\begin{lemma}\label{lem:armijo_convex}
Suppose Assumption~\ref{assumption1} is satisfied.
If the inequality $\nabla g({x}_k)^\top t {d}_k + h({x}_k + t {d}_k) - h({x}_k) \leq -T t \|{\eta}_k\|^2$ holds for some $t > 0$, then for any $t' \in (0, t]$, the following inequality also holds:
\begin{equation*}
    \nabla g({x}_k)^\top t' {d}_k + h({x}_k + t' {d}_k) - h({x}_k) 
    \leq - t'T \|{\eta}_k\|^2.
\end{equation*}
\end{lemma}
\begin{proof}
Let $t' \in (0, t]$ and define $\lambda = {t'}/{t}$, which implies $\lambda \in (0, 1]$. 
We can express ${x}_k + t' {d}_k$ as a convex combination of ${x}_k$ and ${x}_k + t {d}_k$ as follows:
\begin{align*}
    {x}_k + t' {d}_k 
    = (1-\lambda){x}_k + \lambda({x}_k + t {d}_k).
\end{align*}
Since $h$ is convex, we have
\begin{align*}
    h({x}_k + t' {d}_k) 
    &= h\left((1-\lambda){x}_k + \lambda({x}_k + t {d}_k)\right) \\
    &\leq  h({x}_k) + \lambda\left(h({x}_k + t {d}_k) - h({x}_k)\right).
\end{align*}
Subtracting $h({x}_k)$ from both sides and substituting $\lambda = {t'}/{t}$, we obtain:
\begin{equation*}
    h({x}_k + t' {d}_k) - h({x}_k) 
    \leq \frac{t'}{t} (h({x}_k + t {d}_k) - h({x}_k)). 
\end{equation*}
Then considering the assumed inequality, we obtain
\begin{align*}
    \nabla g({x}_k)^\top t' {d}_k + h({x}_k + t' {d}_k) - h({x}_k) 
    &\leq \nabla g({x}_k)^\top t' {d}_k + \frac{t'}{t} (h({x}_k + t {d}_k) - h({x}_k)) \\
    &= \frac{t'}{t} \left( \nabla g({x}_k)^\top t {d}_k + h({x}_k + t {d}_k) - h({x}_k) \right)\\
    &\leq -t'T\|{\eta}_k\|^2.
\end{align*}
\end{proof}
The following lemma shows that the step size $\alpha_k$ is bounded away from zero.
\begin{lemma}\label{lem_step_existence}
Suppose Assumption~\ref{assumption1} holds. 
When Step~4.1 is selected in Algorithm~\ref{alg:proposed}, there exists a step size $\alpha_k$ that satisfies the line search condition~\eqref{linesearch_condition_for_proposed}. 
Furthermore, it holds that
\[
    \overline{\alpha} \equiv \tau \min \left\{ \alpha_{\mathrm{init}},\frac{2(T-\delta)}{LP^2} \right\} \leq \alpha_k \leq \alpha_{\mathrm{init}} \ (\leq 1).
\]
\end{lemma}
\begin{proof}
When Step~4.1 is selected, the condition 
$\nabla g({x}_k)^\top t {d}_k + h({x}_k + t{d}_k) - h({x}_k) \leq -tT  \|{\eta}_k\|^2$ holds. 
Let us consider any $\alpha$ satisfying $0 < \alpha \leq \min \left\{ t, \frac{2(T-\delta)}{LP^2} \right\}$. 
Since $\alpha \leq t$, Lemma~\ref{lem:armijo_convex} guarantees that
\[
    \nabla g({x}_k)^\top \alpha {d}_k + h({x}_k + \alpha{d}_k) - h({x}_k) \leq -\alpha T  \|{\eta}_k\|^2.
\]
Furthermore, it follows from~\eqref{lem_dk_formula} that
\[
    \frac{\alpha L}{2} \| {d}_k\|^2
    \leq \frac{\alpha L P^2}{2} \| {\eta}_k\|^2 \leq (T-\delta)\| {\eta}_k\|^2.
\]
Therefore, by substituting these inequalities into~\eqref{lem_g_lipchitz_formula}, we obtain
\[
    f({x}_k+\alpha {d}_k) - f({x}_k)
    \leq -\alpha \delta \|{\eta}_k\|^2,
\]
which implies that the line search condition is satisfied. 
Since the backtracking strategy is employed with $\tau\in(0,1)$ and $\alpha_{\mathrm{init}} = t$, we obtain
\[
    \overline \alpha \equiv \tau \min \left\{ \alpha_{\mathrm{init}}, \frac{2(T-\delta)}{LP^2} \right\}
    \leq \alpha_k \leq \alpha_{\mathrm{init}} \ (\leq 1).
\]
\end{proof}

Finally, we show that any accumulation point of $\{ {x}_k \}$ is a stationary point of~\eqref{objective_function}.
\begin{theorem}\label{theorem:proposed_alg_global_covergence}
Suppose Assumption \ref{assumption1} is satisfied. Let the sequences $\{ {x}_k \}$ be generated by Algorithm~\ref{alg:proposed}. If the objective function $f$ is bounded below, then we have
\begin{equation}\label{theorem:proposed_alg_global_covergence_formula}
    \lim_{k \to \infty}\|{\eta}_k\|=0.
\end{equation}
Furthermore, if $\{ {x}_k \}$ is bounded, then any accumulation point of $\{ {x}_k \}$ is a stationary point of~\eqref{objective_function}.
\end{theorem}
\begin{proof}
First, we show~\eqref{theorem:proposed_alg_global_covergence_formula}, if Step~4.1 is selected. 
It follows from Lemma~\ref{lem_step_existence} and~\eqref{linesearch_condition_for_proposed_2} that
\begin{equation*}
    f({x}_k + \alpha_k {d}_k) - f({x}_k) 
    \leq - \delta \overline{\alpha} \|{\eta}_k\|^2.
\end{equation*}
On the other hand, if Step~4.2 is selected, from~\eqref{theorem:PGM_decrease_condition_formula}, we obtain 
\[
f({x}_{k+1}) - f ({x}_k) \leq -\frac{\mu}{2}\|{\eta}_k\|^2.
\]
Therefore, there exists a constant $\overline{c}$ such that for all $k$,
\begin{equation}\label{decrease_objective_function}
    f({x}_{k+1}) - f ({x}_k) 
    \leq -\overline{c} \|{\eta}_k\|^2.
\end{equation}
Summing the above inequality from $k=0$ to $k=\tilde{k}$, we have
\begin{equation*}
    f({x}_{\tilde{k}+1}) - f({x}_0)
    \leq f({x}_1) - f({x}_0) + \ldots 
    + f({x}_{\tilde{k}+1}) - f({x}_{\tilde{k}})
    \leq \sum_{k=0}^{\tilde{k}}-\overline{c} \|{\eta}_k\|^2.
\end{equation*}
Since the objective function is bounded below, there exists a constant $\tilde{f}$ such that $\tilde{f} \leq f({x}_k)$. Thus, taking limit $\tilde{k} \to \infty$, we get
\begin{equation*}
    \sum_{k=0}^{\infty} \|{\eta}_k\|^2
    \leq \frac{1}{\overline{c}} (f({x}_0)
    -\tilde{f} ) < \infty.
\end{equation*}
Therefore, we obtain 
\begin{equation*}
    \lim_{k \to \infty} \| {\eta}_k\| =0.
\end{equation*}
Next, we show any accumulation point of $\{{x}_k\}$ is a stationary point of~\eqref{objective_function}. 
Let $\bar{{x}}$ be an accumulation points of $\{ {x}_k \}$.
From~\eqref{eta_k} and~\eqref{theorem:proposed_alg_global_covergence_formula}, we have
\begin{equation*}
     \lim_{k \to \infty} \| {x}_k-{x}_k^+\|
     = \lim_{k \to \infty} \mu\| {\eta}_k\| 
     = 0.
\end{equation*}
Therefore, from~\eqref{lem_mu_optimality_formula} in Lemma~\ref{lem_mu_optimality}, we obtain
\begin{equation*}
    {0} \in \nabla g(\bar{{x}}) + \partial h (\bar{{x}}).
\end{equation*}
This completes the proof of Theorem~\ref{theorem:proposed_alg_global_covergence}. 
\end{proof}

If the objective function is convex, then any stationary point is a global minimizer.
Moreover, if the objective function is strongly convex,  then the global minimizer is unique.
Therefore, we have Corollary~\ref{col:alg_proposed} by Theorem~\ref{theorem:proposed_alg_global_covergence}.
\begin{corollary}\label{col:alg_proposed}
    Suppose Assumption~\ref{assumption1} holds, and let the sequences \( \{ {x}_k \} \) be generated by Algorithm~\ref{alg:proposed}.
    Then, 
    \renewcommand{\theenumi}{\roman{enumi}}
    \begin{enumerate}
        \item if the objective function \( f \) is convex and bounded below, then the relation~\eqref{theorem:proposed_alg_global_covergence_formula} holds. 
        Moreover, if the sequences \( \{ {x}_k \} \) is bounded, then every accumulation point of \( \{ {x}_k \} \) is a global minimizer of problem~\eqref{objective_function},
        \item if, in addition, \( f \) is strongly convex, then the sequence \( \{ {x}_k \} \) converges to the unique global minimizer of problem~\eqref{objective_function}.
    \end{enumerate}
\end{corollary}
\section{Extension to the case where \texorpdfstring{$h$}{h} is weakly convex functions}
\label{sec:weakly_convex}
While the previous section focused on the case where $h$ is convex, many practical applications involve nonsmooth terms that are weakly convex. 
In this section, we extend the proposed method to the weakly convex setting and establish both its global convergence and convergence rate.
Accordingly, we discuss the case where $h : \mathbb{R} \to \mathbb{R} \cup \{+\infty\}$ is a proper lower semicontinuous $\rho$-weakly convex function in \eqref{objective_function}.
For this purpose, we consider the modification of the line search under the weak convexity assumption and the global convergence of the proposed algorithm.
First, we introduce the definition of a weakly convex function.
\begin{definition}[Weakly convex function]
A function $h:\mathbb{R}^n\to{\mathbb{R}}$ is said to be \emph{$\rho$-weakly convex} for $\rho \ge 0$ if
\[
h(x)+\frac{\rho}{2}\|x\|^2
\]
is a convex function.
\end{definition}
In particular, when $\rho=0$, $h$ reduces to a convex function. 
The class of weakly convex functions is particularly important in machine learning and signal processing. 
Typical examples of weakly convex functions include nonconvex sparsity-inducing regularizers such as the smoothly clipped absolute deviation (SCAD) penalty~\cite{Fan2001} and the minimax concave penalty (MCP)~\cite{Zhang2010}. 
By relaxing the convexity assumption, the proposed algorithm can be applied to a wider range of practical nonconvex optimization problems.
Recall that the proximal mapping in~\eqref{definition:proximal_operator_formula} is defined for a convex function~$h$ with any $\mu > 0$. 
Under $\rho$-weak convexity of $h$, the subproblem of $\operatorname{prox}_{\mu h}(v)$ remains strongly convex for $\mu \in (0,1/\rho)$, ensuring that the proximal mapping is well-defined and single-valued~\cite{Bayram2016,Bohm2021}.
Consider the proximal gradient method that generates the sequence $\{x_k\}$ via~\eqref{proximal_operator_4}.
Provided $\mu_k$ satisfies~\eqref{PGM:decrease_condition} and $\mu_k \in (0,1/\rho)$, then
\begin{equation*}
f(x_{k+1}) \leq f(x_k) - \left(\frac{1}{\mu_k}-\rho\right) \|x_{k+1} - x_k\|^2.
\end{equation*}
This inequality is obtained by adapting the constant step size analysis in the proof of Theorem 5.1 in~\cite{Bohm2021} to the variable step size setting, and ensures that $\{f(x_k)\}$ is nonincreasing, which is fundamental for convergence analysis in the nonconvex setting.
It should be noted that, unlike the convex case, the subdifferential of a weakly convex function is no longer the convex subdifferential.
To discuss optimality conditions in this setting, we introduce the Fréchet subdifferential.
\begin{definition}[Fréchet subdifferential]
Let $h:\mathbb{R}^n \to \mathbb{R}$ be a function and let $x \in \mathbb{R}^n$ be a point such that $h(x)$ is finite.
The \emph{Fréchet subdifferential} of $h$ at $x$, denoted by $\hat{\partial} h(x)$, is defined as the set of all vectors $v \in \mathbb{R}^n$ satisfying
\[
h(y) \ge h(x) + \langle v,\, y - x \rangle + o(\|y - x\|)
\quad \text{as } y \to x.
\]
\end{definition}
Throughout this section, $\hat\partial h(x)$ denotes the Fréchet subdifferential.
In the weakly convex setting, we adopt the notion of stationarity defined via the Fréchet subdifferential instead of \eqref{definition:set_stationary_points}.
A point $x \in \mathbb{R}^n$ is said to be a Fréchet stationary point of problem~\eqref{objective_function} if it satisfies $0 \in \nabla g(x) + \hat{\partial} h(x)$, 
which follows from the generalized Fermat's rule and the subdifferential sum rule~\cite{Rockafellar1998}.
Thus, the set of Fréchet stationary points,
$\mathrm{zer}\,\hat{\partial} f$, is defined as follows:
\begin{equation}\label{eq:frechet_stationary_set} 
\mathrm{zer}\,\hat{\partial} f
= \left\{ x \in \mathbb{R}^n \mid
0 \in \nabla g(x) + \hat{\partial} h(x) \right\}.
\end{equation}
It is known that this set $\mathrm{zer}\,\hat{\partial} f$ coincides with the fixed-point set of the iteration defined via the proximal mapping.
Under the weakly convex setting, this relationship is given by the following proposition~\cite{Khanh2025}.
\begin{proposition}\label{prop:fixed_point_characterization}
Suppose that Assumption~\ref{assumption1} holds and that $h$ is a $\rho$-weakly convex function.
Then, for any constant $\mu \in (0, 1/\rho)$, the following equivalence holds:
\begin{equation*}\label{eq:fixed_point_equivalence}
x = \operatorname{prox}_{\mu h}
\bigl(x - \mu\nabla g(x)\bigr)
\iff
x \in \mathrm{zer}\,\hat{\partial} f.
\end{equation*}
\end{proposition}
Based on the above discussion, we present Algorithm~\ref{alg:proposed_weakly} for the weakly convex case.
\begin{algorithm}[htbp]
    \caption{Proximal nonlinear conjugate gradient method when $h$ is weakly convex} \label{alg:proposed_weakly}
    \begin{algorithmic}
        \State \textbf{Step 0:} 
        Given an initial point ${x}_0 \in \textnormal{dom}\:f$, and parameters $\mu_{-1}\in(0,1/\rho)$, $\kappa\in(0,1)$, $\bar{\nu}>0$, $\delta\in(0,1)$, $T>\delta$, $\theta\in(0,1)$, $\bar t\in(0,1)$, $\tau\in(0,1)$. Set $k=0$.
        \State \textbf{Step 1:} 
        Compute ${x}_k^+$ by
        \begin{equation*}
            {x}_k^+={\rm prox}_{\mu_k h}({x}_k-\mu_k\nabla g({x}_k)),
        \end{equation*}
        where $\mu_k$ is the largest value among $\{\mu_{k-1},\mu_{k-1}\kappa,\mu_{k-1}\kappa^2,\cdots\}$ that satisfies satisfies~\eqref{proposed_decrease_condition}.
        \State \textbf{Step 2:} If the stopping condition is satisfied, terminate the algorithm.
        \State \textbf{Step 3:} Compute the search direction ${d}_k$ by (\ref{d_k}).
        \State \textbf{Step 4:}
        Find the largest trial step size $t \in \{1, \theta, \theta^2, \dots \}$ that satisfies condition~\eqref{conditional_branch}.
        If such a $t$ exists and satisfies $t > \bar{t}$, set $\alpha_{\mathrm{init}} = t$ and go to Step~4.1; otherwise, go to Step~4.2.
        
        \State \textbf{\;Step 4.1:}
        Let $\alpha_k$ be the largest value among $\{\alpha_{\mathrm{init}}, \tau \alpha_{\mathrm{init}}, \tau^2 \alpha_{\mathrm{init}}, \dots \}$ that satisfies both condition~\eqref{conditional_branch} and the Armijo condition~\eqref{linesearch_condition_for_proposed}. 
        If such an $\alpha_k$ exists, update the iterate by:
        \[
            x_{k+1} = x_k + \alpha_k d_k,
        \]
        and go to Step~5; otherwise, go to Step~4.2.
        
        \State \textbf{\;Step 4.2:}
        Set $d_k = -{\eta}_k$, and update the iterate by:
        \[
            x_{k+1} = x_k^+,
        \]
        and go to Step~5.
        \State \textbf{Step 5:} Set $k \leftarrow k+1$ and return to Step 1.
    \end{algorithmic}
\end{algorithm}
Note that the initial step size is chosen such that $\mu_{-1} \in (0, 1/\rho)$. 
Since $\{\mu_k\}$ is non-increasing, $\mu_k \in (0, 1/\rho)$ holds for all $k \ge 0$. 
This ensures that the proximal mapping is well-defined and the subsequent analysis remains valid throughout the iterations.
Before establishing the global convergence for Algorithm~\ref{alg:proposed_weakly}, we discuss the impact of weak convexity on the fundamental lemmas, specifically Lemmas~\ref{lem_yk} and \ref{lem:armijo_convex}.
For a $\rho$-weakly convex function $h$ and $\mu \in (0, 1/\rho)$, the property in Proposition~\ref{prop_prox_nonexpansive} is reformulated as the following inequality following \cite{Bayram2016}:
\begin{equation*}
    \|\operatorname{prox}_{\mu h}(u) - \operatorname{prox}_{\mu h}(v)\| \leq \frac{1}{1-\mu\rho}\|u - v\| , \quad \forall u, v \in \mathbb{R}^n.
\end{equation*}
Consequently, the evaluation in the proof of Proposition~\ref{lem_yk} is slightly modified. Using this Lipschitz continuity, the bound becomes:
\begin{align*}
    \|{y}_{k-1}\| 
    &\leq \frac{1}{\mu} \|{x}_k - {x}_{k-1}\| + \frac{1}{\mu} \|\operatorname{prox}_{\mu h}({x}_k - \mu\nabla g({x}_k)) - \operatorname{prox}_{\mu h}({x}_{k-1} - \mu\nabla g({x}_{k-1}))\| \\
    &\leq \frac{1}{\mu} \|{x}_k - {x}_{k-1}\| + \frac{1}{\mu(1-\mu\rho)} \|({x}_k - \mu\nabla g({x}_k)) - ({x}_{k-1} - \mu\nabla g({x}_{k-1}))\| \\
    &\leq \frac{1}{\mu} \|{x}_k - {x}_{k-1}\| + \frac{1}{\mu(1-\mu\rho)} (\|{x}_k - {x}_{k-1}\| + \mu \|\nabla g({x}_k) - \nabla g({x}_{k-1})\|) \\
    &\leq \left( \frac{1}{\mu} + \frac{1+\mu L}{\mu(1-\mu\rho)} \right) \|{x}_k - {x}_{k-1}\|.
\end{align*}
Thus, the inequality $\|{y}_{k-1}\| \leq K \|{s}_{k-1}\|$ still holds with a new positive constant $K = \frac{1}{\mu} + \frac{1+\mu L}{\mu(1-\mu\rho)}$. 
In addition, since Lemma~\ref{lem:armijo_convex} does not necessarily hold for weakly convex $h$, the algorithm is modified to ensure that conditions~\eqref{conditional_branch} and \eqref{linesearch_condition_for_proposed} are satisfied when determining the step size.
Therefore, the key ingredients required for the convergence analysis are preserved, and we obtain Theorem~\ref{corollary:weakly_convex_convergence}.
\begin{theorem}\label{corollary:weakly_convex_convergence}
Suppose Assumption~\ref{assumption1} is satisfied, and $h$ is a $\rho$-weakly convex function. 
Let the sequence $\{ {x}_k \}$ be generated by Algorithm~\ref{alg:proposed_weakly}. 
If the objective function $f$ is bounded below, then we have
\[
\lim_{k \to \infty}\|{\eta}_k\|=0.
\]
Furthermore, if $\{ {x}_k \}$ is bounded, then any accumulation point of $\{ {x}_k \}$ is a stationary point of problem~\ref{objective_function} in the sense that~\eqref{eq:frechet_stationary_set}.
\end{theorem}
Next, we now establish the convergence rate of Algorithm~\ref{alg:proposed_weakly} when $g$ is nonconvex.
The following theorem shows that both the \textit{forward-backward residual} $\|\eta_k\|$ and the distance to the subdifferential $\operatorname{dist}(0,\partial f(x_k^+))$ achieve an $\mathcal{O}(1/\sqrt{k})$ convergence rate.
\begin{theorem}\label{thm:subgradient_rate}
    Suppose Assumption~\ref{assumption1} is satisfied and $h$ is a $\rho$-weakly convex function. 
    Let the sequence $\{x_k\}$ be generated by Algorithm~\ref{alg:proposed_weakly}. 
    If the objective function $f$ is bounded below by $f^*$, then for any $k \ge 0$, the following assertions hold:
    \begin{enumerate}
        \item[(i)] The stationarity measure $\eta_k$ satisfies
        \begin{equation}\label{eq:eta_bound}
            \min_{0 \le i \le k} \|\eta_i\| \le \frac{1}{\sqrt{\bar{c}(k+1)}} \sqrt{f(x_0) - f^*}.
        \end{equation}
        \item[(ii)] The distance from the origin to the subdifferential $\partial f(x_k^+)$ satisfies
        \begin{equation}\label{eq:subgradient_bound}
            \min_{0 \le i \le k} \operatorname{dist}(0,\partial f(x_i^+)) \le \frac{1 + L_g \mu}{\sqrt{\bar{c}(k+1)}} \sqrt{f(x_0) - f^*},
        \end{equation}
    \end{enumerate}
    where $\bar{c} > 0$ is a positive constant introduced in the proof of Theorem~\ref{theorem:proposed_alg_global_covergence}.
\end{theorem}
\begin{proof}
    From the descent condition~\eqref{decrease_objective_function}, we have
    \begin{equation*}
        \|\eta_i\|^2 \le \frac{1}{\bar{c}} \left( f(x_i) - f(x_{i+1}) \right).
    \end{equation*}
    Summing this inequality from $i = 0$ to $k$ yields
    \begin{equation*}
        \sum_{i=0}^{k} \|\eta_i\|^2 \le \frac{1}{\bar{c}} \left( f(x_0) - f(x_{k+1}) \right) \le \frac{1}{\bar{c}} \left( f(x_0) - f^* \right),
    \end{equation*}
    which implies that
    \begin{equation*}
        \min_{0 \le i \le k} \|\eta_i\|^2 \le \frac{1}{\bar{c}(k+1)} \left( f(x_0) - f^* \right).
    \end{equation*}
    Taking the square root on both sides proves~\eqref{eq:eta_bound}.
    Next, by the optimality condition for the proximal operator $\mathrm{prox}_{\mu h}$, we have
    \begin{equation*}
        \frac{1}{\mu}(x_i - x_i^+) - \nabla g(x_i) \in \partial h(x_i^+).
    \end{equation*}
    Adding $\nabla g(x_i^+)$ to both sides, we define
    \begin{equation*}
        \xi_i := \nabla g(x_i^+) - \nabla g(x_i) + \frac{1}{\mu}(x_i - x_i^+) \in \nabla g(x_i^+) + \partial h(x_i^+) = \partial f(x_i^+).
    \end{equation*}
    It follows that
    \begin{align*}
        \|\xi_i\| &\le \|\nabla g(x_i^+) - \nabla g(x_i)\| + \left\| \frac{1}{\mu}(x_i - x_i^+) \right\| \\
        &= (L_g \mu + 1) \|\eta_i\|.
    \end{align*}
    Squaring both sides and taking the minimum over $0 \le i \le k$, we obtain
    \begin{equation*}
        \min_{0 \le i \le k} \|\xi_i\|^2 \le (1 + L_g \mu)^2 \min_{0 \le i \le k} \|\eta_i\|^2.
    \end{equation*}
    Thus, $\operatorname{dist}(0, \partial f(x_i^+)) \le (1 + L_g \mu) \|\eta_i\|$.
    Taking the minimum over $0 \le i \le k$ and applying~\eqref{eq:eta_bound}, we arrive at
    \begin{equation*}
        \min_{0 \le i \le k} \operatorname{dist}(0,\partial f(x_i^+)) \le (1 + L_g \mu) \min_{0 \le i \le k} \|\eta_i\| \le \frac{1 + L_g \mu}{\sqrt{\bar{c}(k+1)}} \sqrt{f(x_0) - f^*},
    \end{equation*}
    which completes the proof.
\end{proof}
\begin{remark}\label{rmk:subgradient_rate}
Several comments regarding the derived convergence rates in Theorem~\ref{thm:subgradient_rate} are in order:
\begin{enumerate}
    \item[(i)] The $\mathcal{O}(1/\sqrt{k})$ rate for the \textit{forward-backward residual} $\|\eta_k\|$ in~(i) matches the rate for proximal gradient methods applied to the sum of a nonconvex smooth function and a weakly convex function~\cite{Hurault2023}. Furthermore, this rate is consistent with the standard rate established for the sum of a nonconvex function and a convex function in Theorem~10.15 of \cite{Beck2017}.
    \item[(ii)] The $\mathcal{O}(1/\sqrt{k})$ rate for the distance to the subdifferential $\operatorname{dist}(0,\partial f(x_k^+))$ in~(ii) aligns with the rate established by \cite{Bohm2021} for the proximal gradient method applied to the sum of a nonconvex smooth function and a weakly convex function. Furthermore, this rate matches the convergence rate of the steepest descent method for general smooth nonconvex optimization~\cite{bookNesterov2004}.
\end{enumerate}
\end{remark}
Next, we analyze the convergence rate of the Algorithm~\ref{alg:proposed_weakly} under the assumption that \( g \) is strongly convex. 
We establish the following convergence rate result.
\begin{theorem}\label{theorem:convergence_rate}
    Suppose Assumptions~\ref{assumption1} and~\ref{assumption_strong_convex} are satisfied and $m > \rho$ holds.
    Let the sequences \( \{ {x}_k \} \) and \( \{ {x}_k^+ \} \) be generated by Algorithm~\ref{alg:proposed_weakly}.
    Then the sequence converges to the optimal solution ${x}^*$, and the following inequality holds:
    \begin{equation}\label{eq:convergence}
        \min_{0\le i\le k} \|{x}_i-{x}^*\| 
        \leq \sqrt{\frac{f({x}_0)-f({x}^*)}{C(k+1)}},
    \end{equation}
    where $C$ is a constant.
\end{theorem}
\begin{proof}
For any $v_x \in \hat\partial h(x^+)$ and $v_y \in \hat\partial h(y^+)$, the $\rho$-weak convexity of $h$ yields~\cite{Davis2019}:
\[(v_x - v_y)^\top (x^+ - y^+) \ge -\rho \|x^+ - y^+\|^2.
\]
From \eqref{lem_mu_optimality}, it follows that
\begin{equation}\label{R_linear_formula_1}
(\eta_\mu({x}) - \eta_\mu({y}))^\top({x}^+-{y}^+) \geq (\nabla g({x})-\nabla g({y}))^\top({x}^+-{y}^+)-\rho \|x^+ - y^+\|^2.
\end{equation}
Since $g$ is $m$-strongly convex and $h$ is $\rho$-weakly convex, it follows that $f$ is $(m - \rho)$-strongly convex. 
The condition $m > \rho$ ensures that $f$ possesses a unique minimizer $x^*$.
Substituting $x^+ - y^+ = (x-y) - \mu(\eta_\mu(x) - \eta_\mu(y))$ from \eqref{function_eta} into \eqref{R_linear_formula_1} and setting $y = x^*$ with $\eta_\mu(x^*) = 0$ yields
\begin{align}\label{R_linear_formula_2}
    &\eta_\mu({x})^\top ({x}-{x}^*) 
    + \mu\eta_\mu({x})^\top(\nabla g({x})-\nabla g({x}^*))\\
    &\qquad\geq
    (\nabla g({x})-\nabla g({x}^*))^\top ({x}-{x}^*) 
    + \mu\|\eta_\mu({x})\|^2
    -\rho\|x-x^*-\mu \eta_\mu(x)\|^2. \notag
\end{align}
Using the $m$-strong convexity of $g$ and $\|x - x^* - \mu \eta_\mu(x)\|^2 = \|x - x^*\|^2 - 2\mu (x - x^*)^\top \eta_\mu(x) + \mu^2 \|\eta_\mu(x)\|^2$,  \eqref{R_linear_formula_2} can be rewritten as 
\begin{align}\label{R_linear_formula_3}
    (1-2\mu\rho)\eta_\mu({x})^\top ({x}-{x}^*) 
    &+ \mu\eta_\mu({x})^\top(\nabla g({x})-\nabla g({x}^*))\\
    &\quad\geq
    (m-\rho)\|x-x^*\|^2 +\mu(1-\mu\rho)\|\eta_\mu(x)\|^2.\notag
\end{align}
Applying the Cauchy-Schwarz inequality, the Lipschitz continuity of $\nabla g$ in~\eqref{lipschitz_continuous_1}, and the inequality \( ab \leq \frac{1}{2}(a^2 + b^2) \) with $a = \frac{(|1-2\mu\rho|+\mu L)\|\eta_\mu({x})\|}{\sqrt{m-\rho}}$ and $b = \sqrt{m-\rho}\|{x}-{x}^*\|$, the left-hand side of \eqref{R_linear_formula_3} is bounded by
\begin{align*}
(1-2\mu\rho)\eta_\mu(x)^\top (x-x^*) 
&+ \mu \eta_\mu(x)^\top(\nabla g(x)-\nabla g(x^*)) \\
&\leq (|1-2\mu\rho|+\mu L)\|\eta_\mu(x)\|\|x-x^*\| \\
&\leq \frac{(|1-2\mu\rho|+\mu L)^2}{2(m-\rho)}\|\eta_\mu(x)\|^2 + \frac{m-\rho}{2}\|x-x^*\|^2,
\end{align*}
where $m > 0$. Combining these results, \eqref{R_linear_formula_3} simplifies to
\begin{equation*}
\frac{m-\rho}{2} \|x-x^*\|^2 \leq \left(\frac{(|1-2\mu\rho|+\mu L)^2}{2(m-\rho)}-\mu(1-\mu\rho)\right)\|\eta_\mu(x)\|^2.
\end{equation*}
In view of~\eqref{decrease_objective_function}, defining $C = \frac{\bar{c}(m-\rho)}{{2} \left( \frac{(|1-2\mu\rho|+\mu L)^2}{2(m-\rho)}-\mu(1-\mu\rho) \right)}$ implies
\begin{equation*}
f(x_k) - f(x_{k+1}) \geq \bar{c}\|\eta_k\|^2 \geq C\|x_k-x^*\|^2.
\end{equation*}
Therefore,
\begin{equation*}
    f({x}_0)-f({x}^*)
    \geq f({x}_0)-f({x}_{k+1})
    \geq C \sum_{i=0}^{k} \|{x}_i-{x}^*\|^2
    \geq C(k+1)\min_{0\le i\le k}\|{x}_i-{x}^*\|^2,
\end{equation*}
which yields~\eqref{eq:convergence}.
\end{proof}
\begin{remark}
When the function $h$ is convex (i.e., $\rho = 0$), the result in Theorem~\ref{theorem:convergence_rate} naturally covers the convergence analysis for Algorithm~\ref{alg:proposed}. In this case, the constant $C$ simplifies to $C = \frac{\bar{c} m}{2 \left( \frac{(1+\mu L)^2}{2m} - \mu \right)}$.
\end{remark}
To the best of our knowledge, theoretical bounds under this specific metric are relatively understudied.

\section{Numerical experiments}\label{sec:numerical_experiment}
In this section, we evaluate the numerical performance of Algorithms~\ref{alg:proposed} and \ref{alg:proposed_weakly}.
Sections~\ref{sec:ne_lasso}--\ref{sec:ne_student_t_l1} evaluate Algorithm~\ref{alg:proposed} on the Lasso problem, the $\ell_1$-regularized logistic regression problem, and the $\ell_1$-regularized Student's $t$-regression problem.
Section~\ref{sec:ne_ls_MCP} evaluates Algorithm~\ref{alg:proposed_weakly} on MCP-regularized least squares problems, where $h$ is weakly convex. 
Note that in Sections~\ref{sec:ne_lasso}--\ref{sec:ne_logistic_l1}
$g$ and $h$ are convex, in Section~\ref{sec:ne_student_t_l1} $g$ is nonconvex and $h$ is convex, and in Section~\ref{sec:ne_ls_MCP} 
$g$ is convex and $h$ is weakly convex. 
For comparison, we consider solvers from the TFOCS package~\cite{tfocs}, which provides various implementations of the proximal gradient method and its accelerated variants, as well as the PNOPT package~\cite{pnopt}, which provides a solver for proximal Newton-type methods.
Specifically, from the TFOCS package, we employ TFOCS-GRA (the standard proximal gradient method) and default solver, denoted as PGM and TFOCS, respectively.
We also use the default settings for the PNOPT package, denoted as PNOPT.
It should be noted that, following the literature, PGM and TFOCS employ a strategy that slightly increases the step size $\mu_k$ (e.g., $\mu_{k} = \mu_{k-1}/0.9$) at each iteration~\cite{Becker2011}.
All numerical experiments were conducted in MATLAB R2024a on a machine running Windows 11 Pro, equipped with an Intel Core i7 processor (3.2 GHz) and 32 GB of RAM.

In addition to Algorithm~\ref{alg:proposed}, we consider a variant, denoted by Algorithm~3', which incorporates quadratic interpolation into the line search procedure.
Let $\alpha_k^{(i)}$ and $\tau_k^{(i)}\in (0,1)$ denote the step size and the step size scaling factor at the $i$th line search iteration, respectively, and define the function $\phi_k(\alpha) = f(x_k + \alpha d_k)$.
When the Armijo condition~\eqref{linesearch_condition_for_proposed} is not satisfied, the step size scaling factor by quadratic interpolation $\tau_{k,\mathrm{quad}}^{(i)}$ is computed as
\[
\tau_{k,\mathrm{quad}}^{(i)}
=
-\frac{\phi_k'(0)\alpha_k^{(i)}}
{2\bigl(\phi_k(\alpha_k^{(i)})-\phi_k(0)-\phi_k'(0)\alpha_k^{(i)}\bigr)}.
\]
Here, $\phi_k(0) = f(x_k)$ and, since $h$ may be nondifferentiable, we employ
$\phi_k'(0) = \nabla g(x_k)^\top d_k + h'(x_k; d_k)$,
where $h'(x_k; d_k)$ denotes the directional derivative of $h$ at $x_k$ along the direction $d_k$.
To ensure that the updated step size lies in $(0,\alpha_k^{(i)})$,
we restrict the scaling factor to $[10^{-8},\,0.99]$.
More precisely, we set 
\[
\tau_k^{(i)} =
\min\left\{
0.99,\,
\max\left\{10^{-8},\,
\tau_{k,\mathrm{quad}}^{(i)}
\right\}
\right\}
\]
and update the step size by $\alpha_k^{(i+1)} = \tau_k^{(i)} \alpha_k^{(i)}.$
Algorithm~3' retains the convergence properties of Algorithm~\ref{alg:proposed}  since it replaces $\tau\in(0,1)$ with $\tau_k^{(i)}\in [10^{-8},0.99]$.
We note that employing quadratic interpolation within the line search is a well-known standard technique in nonlinear conjugate gradient methods \cite{bookNocedal2006, bookGrippo2023}. 
By incorporating local derivative information, quadratic interpolation provides a better step size estimate and reduces the number of function evaluations per iteration, thereby accelerating practical convergence.

\subsection{LASSO problem}\label{sec:ne_lasso}
In this section, we evaluate the numerical performance of the proposed Algorithm~\ref{alg:proposed} by solving LASSO problem formulated as follows:
\begin{equation}\label{lasso}
    \min_{{x} \in \mathbb{R}^n} \| A{x} - {b} \|^2 
    + \lambda \|{x}\|_1
\end{equation}
where $m$ denotes the number of data samples, $n$ the number of features, and $\lambda$ the regularization parameter. 
The matrix $A \in \mathbb{R}^{m \times n} $ is randomly generated, with each element sampled independently from the uniform distribution over the interval $[0,1)$. 
The vector ${b}$ is generated according to
\begin{align*}
    {b} = A\tilde{{x}} + 0.01{\epsilon},
\end{align*}
where \(\tilde{{x}} \in \mathbb{R}^n\) is a sparse vector with $s>0$ components randomly set to 1 and the remaining entries set to 0.
The noise vector ${\epsilon} \in \mathbb{R}^m$ has entries independently drawn from the standard normal distribution. 
In Algorithm~\ref{alg:proposed}, we set the parameters as follows: $\mu_{-1}=1$, $\kappa=1/2$, $\bar{\nu}=10^{-8}$, $\delta =10^{-4}$, $T=10^{-3}$, $\theta=1/2$, $\bar{t}={1}/{2^{20}}$, and $\tau=1/2$. 
In these experiments, we used seven datasets listed in Table~\ref{tab:lasso_data}.
\begin{table}[htbp]
\centering
\caption{Problem scale settings $(m, n, s)$}
\label{tab:lasso_data}
\begin{tabular}{|l|l|}
\hline
Category & $(m, n, s)$ \\ \hline
$m < n$  & (500, 550, 50), (1000, 1050, 50), \\
$m > n$  & (500, 150, 30), (1000, 300, 60), (3000, 500, 180), \\
         & (5000, 1500, 300), (7000, 2000, 400) \\ \hline
\end{tabular}
\end{table}
For the large-scale problem with \( (m, n, s) = (7000, 2000, 400) \), $50\%$ of the elements of the matrix A are set to 0, resulting in a sparse matrix. 
Algorithm~\ref{alg:proposed}, Algorithm~3' and PNOPT are terminated when \( \frac{\|{x}_{k}^+ - {x}_k\|}{\textnormal{max}\{1,\|{x}_k\|\}} \leq 10^{-8} \).
PGM and TFOCS are terminated when \( \frac{\|{x}_{k+1} - {x}_k\|}{\textnormal{max}\{1,\|{x}_k\|\}} \leq 10^{-8} \), following their default settings.
For each type of dataset, 10 random instances were created, and experiments were conducted for each instance. 
The initial point was set to \( {x}_0 = (0,...,0)^\top \). 
We demonstrate the effectiveness of our proposed algorithm by the performance profiles of Dolan and Moré~\cite{Dolan2002}.
\begin{quote}
For $n_s$ solvers and $n_p$ problems, the performance profiles $P : \mathbb{R}^n \to [0, 1]$ are defined as follows.
Consider sets $\mathcal{P}_r$ and $\mathcal{S}$, which represent the set of problems and solvers, respectively.
For each problem $p \in \mathcal{P}_r$ and each solver $s \in \mathcal{S}$, $t_{p,s}$ is defined as the CPU time required for solver $s$ to solve problem $p$. 
The performance ratio is given by $r_{p,s} = \frac{t_{p,s}}{\min_s t_{p,s}}$. 
The performance profile is then defined as $P(\tau) = \frac{1}{n_p} \left| \left\{ p \in \mathcal{P}_r \,\middle|\, r_{p,s} \leq \tau \right\} \right|$ 
for all $\tau \geq 1$. 
Here, $P(\tau)$ represents the probability that a solver in $\mathcal{S}$ achieves a performance ratio within a factor $\tau$ of the best performance.
\end{quote}
The numerical results are summarized in Tables~\ref{table:computation_time_lasso} and \ref{table:iteration_count_lasso}, and the corresponding performance profiles are illustrated in Figures~\ref{fig:all_lamb01} and \ref{fig:all_lamb001} for $\lambda = 0.1$ and $\lambda = 0.01$, respectively.
\begin{table}[htbp]
\centering
\caption{CPU time (s)}
\label{table:computation_time_lasso}
\begin{tabular}{ccccccc}
\toprule
$\lambda$ & $(m,n,s)$ & Algorithm~\ref{alg:proposed} & Algorithm 3' & PGM & TFOCS & PNOPT \\
\midrule
0.1  & (500, 550, 50)    & 5.7   & \textbf{1.7}  & 71.9  & 29.0  & 14.6 \\
0.1  & (1000, 1050, 50)  & 70.3  & \textbf{7.2}  & 468.1 & 120.0 & 17.5 \\
0.1  & (500, 150, 50)    & 0.3   & \textbf{0.1}  & 2.9   & 3.3   & 1.4 \\
0.1  & (1000, 300, 60)   & 2.0   & \textbf{0.5}  & 10.4  & 7.3   & 3.7 \\
0.1  & (3000, 500, 180)  & 11.9  & \textbf{1.8}  & 52.7  & 42.2  & 3.3 \\
0.1  & (5000, 1500, 300) & 170.9 & \textbf{19.3} & 1733.3 & 519.6 & 37.4 \\
0.1  & (7000, 2000, 400) & 160.5 & 21.0 & 442.5 & 249.0 & \textbf{8.8} \\
\midrule
0.01 & (500, 550, 50)    & 22.8  & \textbf{3.8}  & 635.5     & 71.2   & 141.6 \\
0.01 & (1000, 1050, 50)  & 162.6 & \textbf{36.5} & 3747.6 & 336.7 & 336.3 \\
0.01 & (500, 150, 50)    & 0.3   & \textbf{0.1}  & 3.2   & 2.8   & 1.7 \\
0.01 & (1000, 300, 60)   & 3.4   & \textbf{0.4}  & 16.5  & 12.1  & 2.3 \\
0.01 & (3000, 500, 180)  & 13.5  & \textbf{0.8}  & 52.8  & 47.7  & 1.8 \\
0.01 & (5000, 1500, 300) & 190.6 & 43.4 & 1219.8 & 415.5 & \textbf{4.8} \\
0.01 & (7000, 2000, 400) & 175.2 & 16.8 & 599.1 & 404.1 & \textbf{5.0} \\
\bottomrule
\end{tabular}
\end{table}
\begin{table}[htbp]
\centering
\caption{Number of iterations}
\label{table:iteration_count_lasso}
\begin{tabular}{c c c c c c c}
\toprule
$\lambda$ & $(m,n,s)$ & Algorithm~\ref{alg:proposed} & Algorithm 3' & PGM & TFOCS & PNOPT \\
\midrule
0.1  & (500, 550, 50)   & 1198.0 & 890.2 & 52193.4 & 12858.8 & \textbf{149.0} \\
0.1  & (1000, 1050, 50)  & 1746.1 & 882.2  & 104524.8 & 15278.8 & \textbf{239.5} \\
0.1  & (500, 150, 50)   & 333.8  & 232.4 & 3630.6  & 3029.2  & \textbf{46.6} \\
0.1  & (1000, 300, 60)  & 414.8  & 287.5 & 6726.0  & 3545.2  & \textbf{48.3} \\
0.1  & (3000, 500, 180) & 296.2  & 177.3 & 10059.3 & 6073.4  & \textbf{104.5} \\
0.1  & (5000, 1500, 300)& 494.2  & \textbf{274.5} & 39599.8 & 8520.9  & 832.7 \\
0.1  & (7000, 2000, 400) & 298.4  & 175.2  & 8789.4  & 3661.3  & \textbf{83.6} \\
\midrule
0.01 & (500, 550, 50)    & 2375.9 & \textbf{1652.4} & 454675.5   & 34017.6   & 5800.7 \\
0.01 & (1000, 1050, 50)  & 3310.6 & 2281.6 & 812714.4 & 39423.1 & \textbf{239.5} \\
0.01 & (500, 150, 50)    & 243.0  & 140.3  & 5479.8  & 3682.3  & \textbf{47.7} \\
0.01 & (1000, 300, 60)   & 303.8  & 163.7  & 10707.2 & 5279.6  & \textbf{61.4} \\
0.01 & (3000, 500, 180)  & 175.5  & 54.6   & 11384.8 & 7042.9  & \textbf{38.7} \\
0.01 & (5000, 1500, 300) & 363.2  & 408.0  & 45033.8 & 11007.5 & \textbf{45.0} \\
0.01 & (7000, 2000, 400) & 201.0  & 105.7  & 10353.0 & 4951.4  & \textbf{45.0} \\
\bottomrule
\end{tabular}
\end{table}
\begin{figure}[htbp]
    \centering
    \begin{subfigure}[b]{0.45\textwidth}
        \centering
        \includegraphics[width=\textwidth]
        {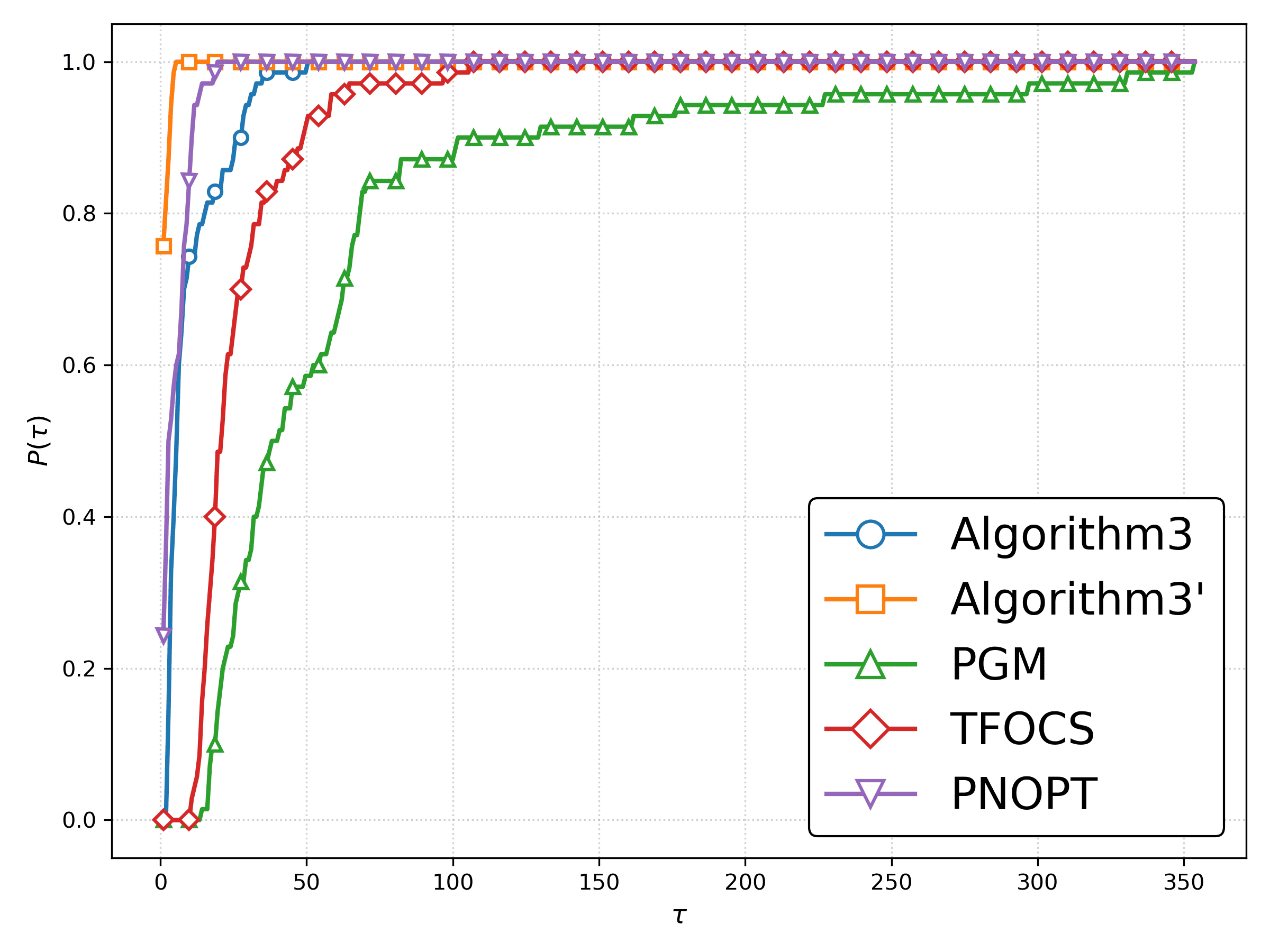} 
        \caption{CPU Time}
    \end{subfigure}
    \hfill
    \begin{subfigure}[b]{0.45\textwidth}
        \centering
        \includegraphics[width=\textwidth]
        {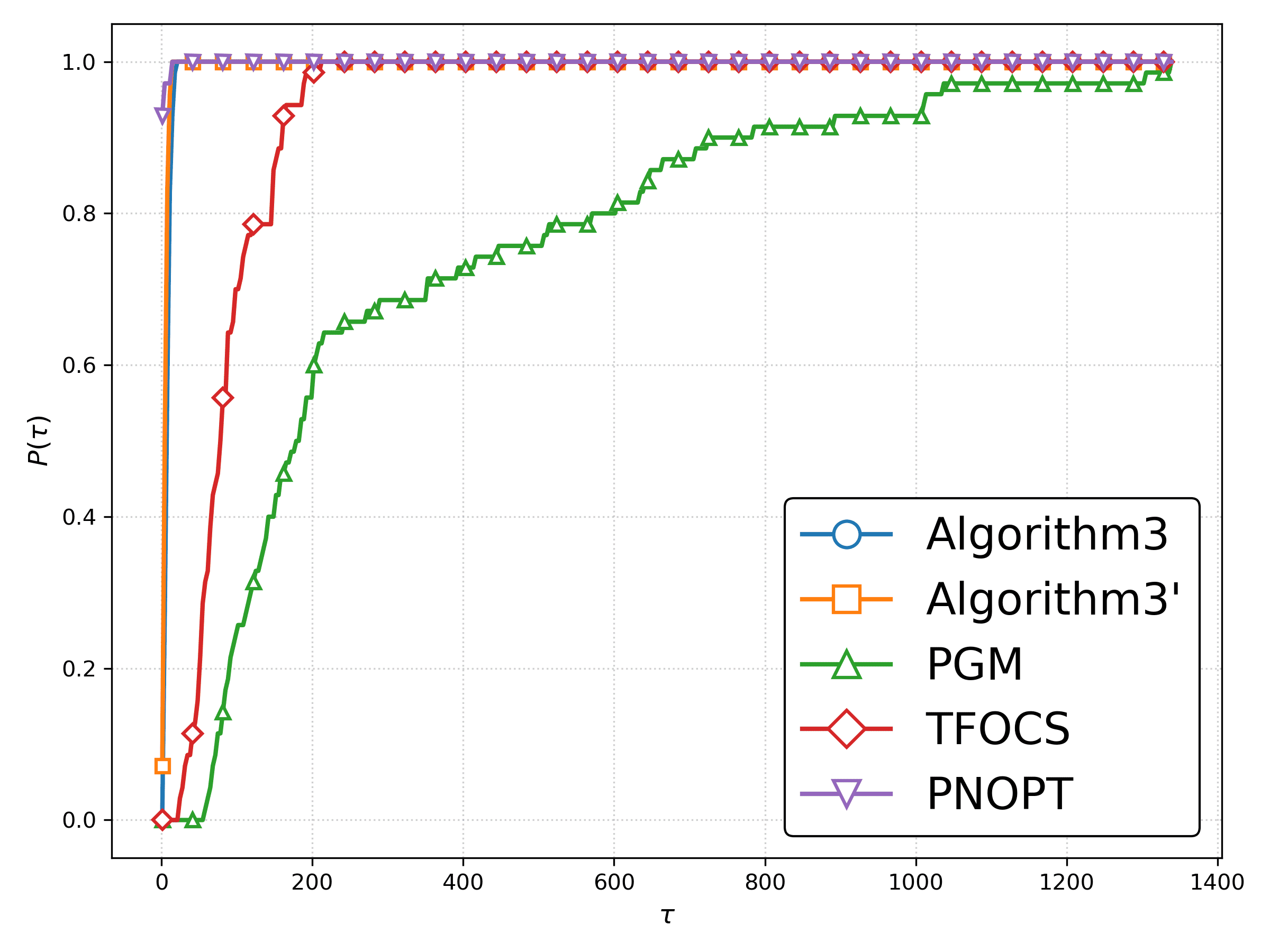} 
        \caption{Iteration Count}
    \end{subfigure}
    \caption{Performance profiles for $\lambda = 0.1$ over all datasets}
    \label{fig:all_lamb01}
\end{figure}
\begin{figure}[htbp]
    \centering
    \begin{subfigure}[b]{0.45\textwidth}
        \centering
        \includegraphics[width=\textwidth]
        {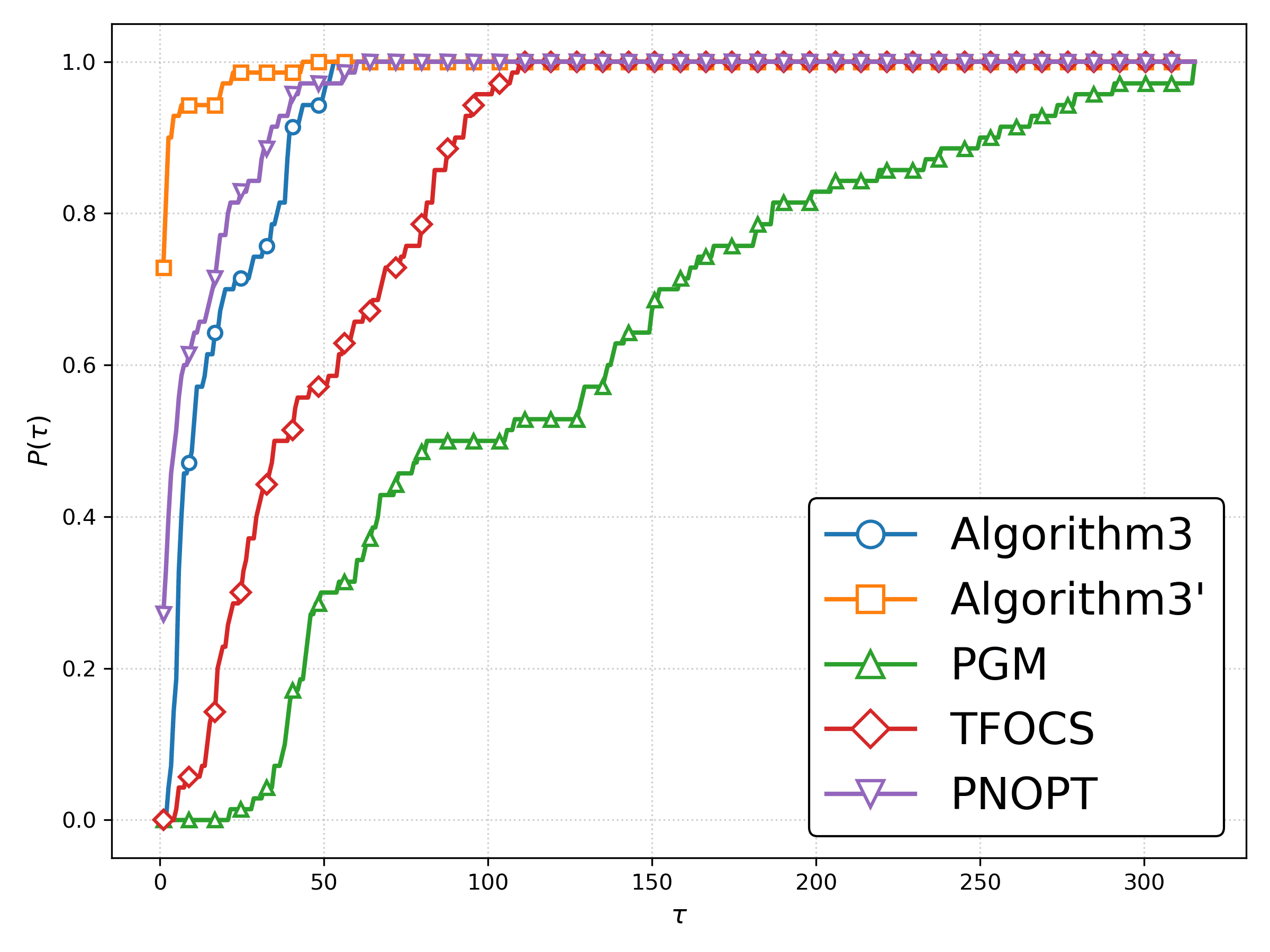} 
        \caption{CPU Time}
    \end{subfigure}
    \hfill
    \begin{subfigure}[b]{0.45\textwidth}
        \centering
        \includegraphics[width=\textwidth]
        {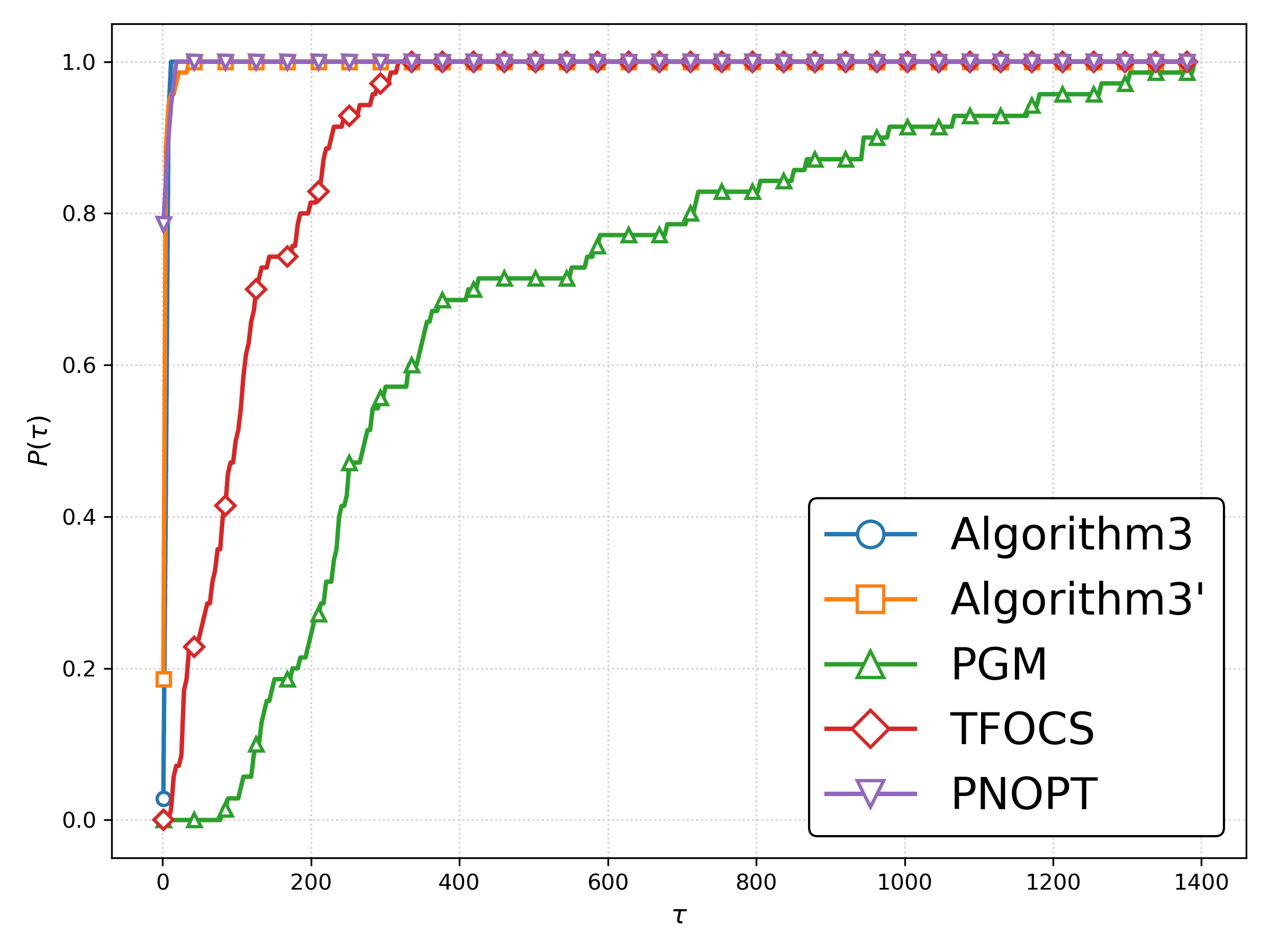} 
        \caption{Iteration Count}
    \end{subfigure}
    \caption{Performance profiles for $\lambda = 0.01$ over all datasets}
    \label{fig:all_lamb001}
\end{figure}
From these results, we observe that Algorithm~3' achieves the best performance in terms of CPU time for most datasets. 
In particular, it significantly outperforms first-order methods such as PGM and TFOCS.
In terms of the number of iterations, Algorithm~\ref{alg:proposed} and 3' follow PNOPT as the second most efficient, consistently requiring fewer iterations than PGM and TFOCS.
While PNOPT yields the smallest iteration counts, this is expected as it is a quasi-Newton-type method that incorporates second-order information.
However, since each iteration of PNOPT is computationally expensive, Algorithm~3' achieves better efficiency in terms of CPU time.
Comparing Algorithm~\ref{alg:proposed} and Algorithm~3', the latter demonstrates superior performance in both CPU time and iteration count. 
As anticipated from the literature, incorporating quadratic interpolation effectively refines the step size estimation and reduces the number of line search trials, thereby significantly enhancing the overall computational efficiency.
Furthermore, as the problem size increases, CPU times and the iteration counts of PGM and TFOCS increase significantly, whereas Algorithm~\ref{alg:proposed} and 3' maintain relatively stable performance.

Next, we examine the convergence behavior of each method.
Figures~\ref{fig:lasso_lamb01_convergence_500_550} and 
\ref{fig:lasso_lamb01_convergence_3000} present the convergence results for representative instances with $(m,n,s)=(500,550,50)$ and $(m,n,s)=(3000,500,180)$, respectively, both using $\lambda=0.1$.
\begin{figure}[htbp]
    \centering
    \begin{subfigure}[b]{0.45\textwidth}
        \centering
        \includegraphics[width=\textwidth]
        {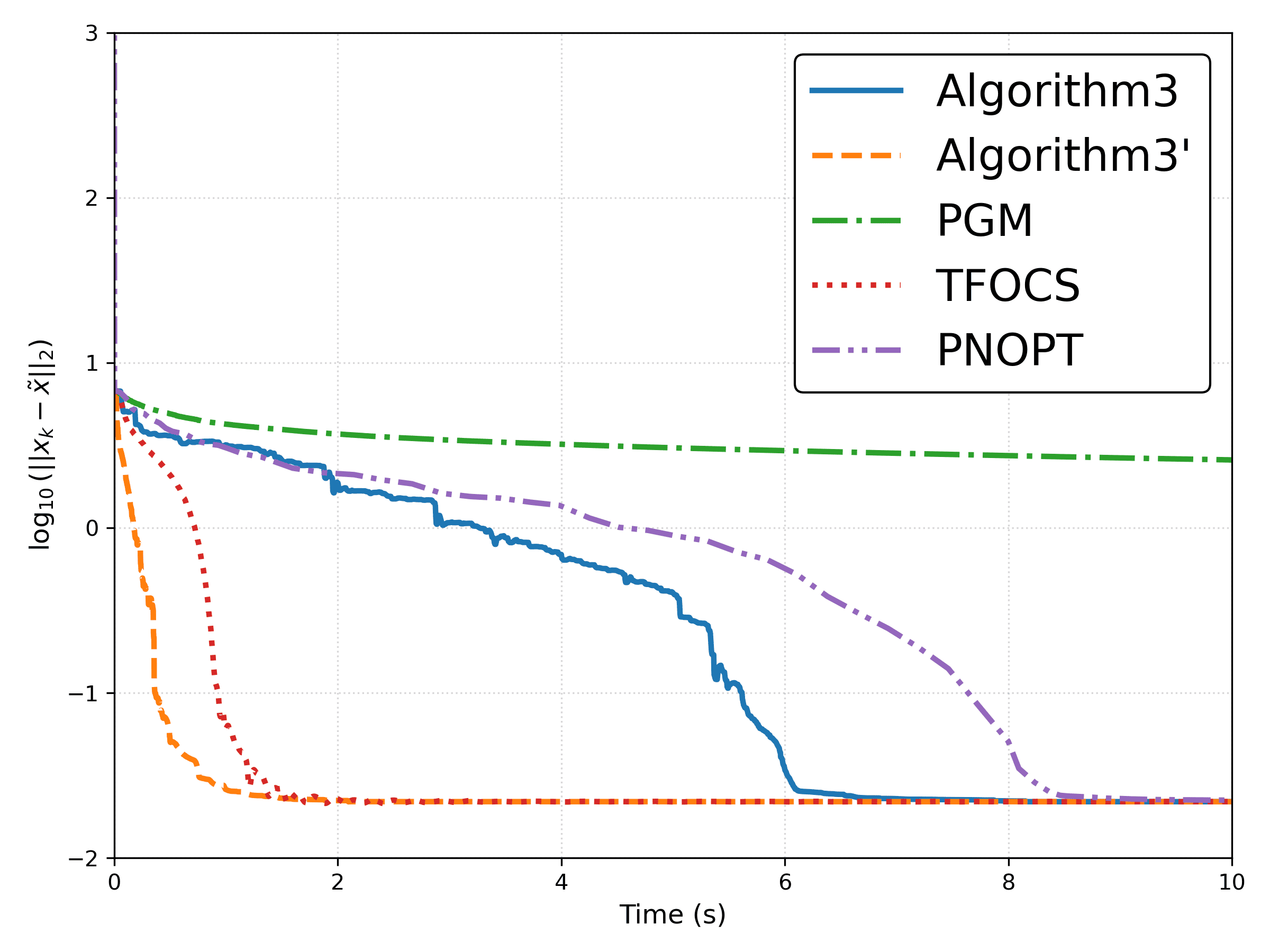}
        \caption{Convergence vs. CPU time}
    \end{subfigure}
    \hfill
    \begin{subfigure}[b]{0.45\textwidth}
        \centering
        \includegraphics[width=\textwidth]
        {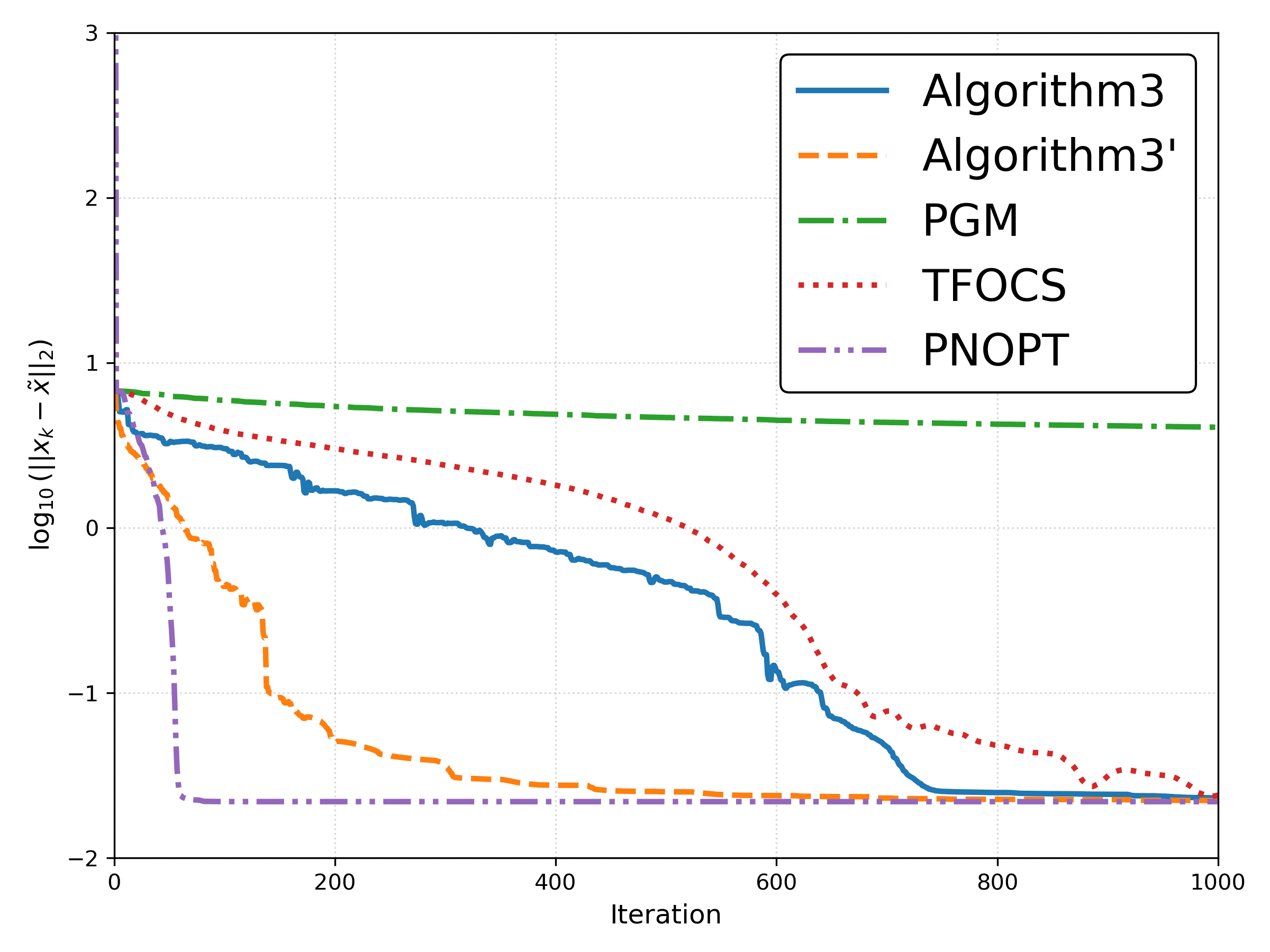}
        \caption{Convergence vs. iteration count}
    \end{subfigure}
    \caption{Convergence behavior for $(m,n,s)=(500,550,50)$ with $\lambda = 0.1$}
    \label{fig:lasso_lamb01_convergence_500_550}
\end{figure}
\begin{figure}[htbp]
    \centering
    \begin{subfigure}[b]{0.45\textwidth}
        \centering
        \includegraphics[width=\textwidth]
        {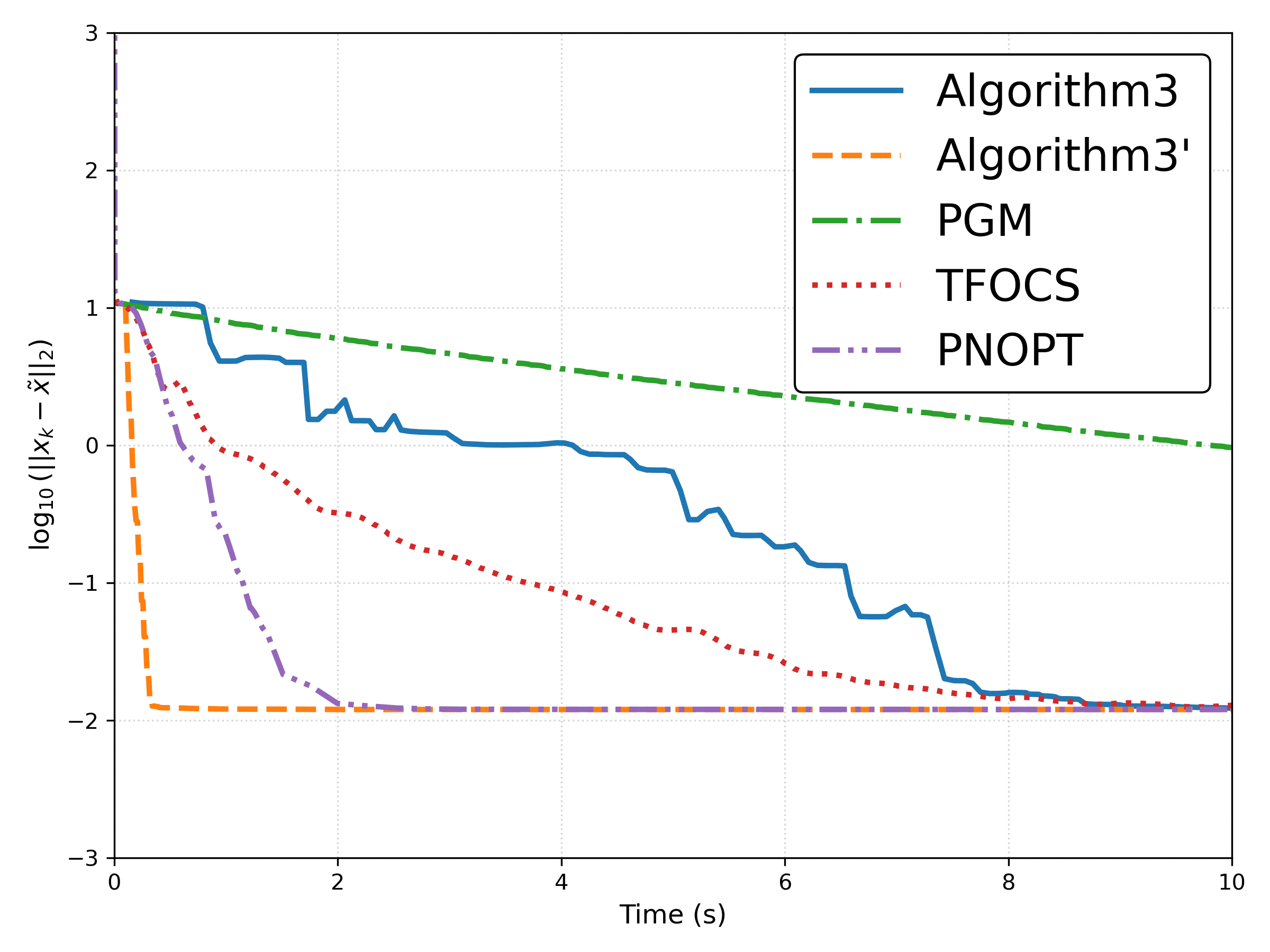}
        \caption{Convergence vs. CPU time}
    \end{subfigure}
    \hfill
    \begin{subfigure}[b]{0.45\textwidth}
        \centering
        \includegraphics[width=\textwidth]
        {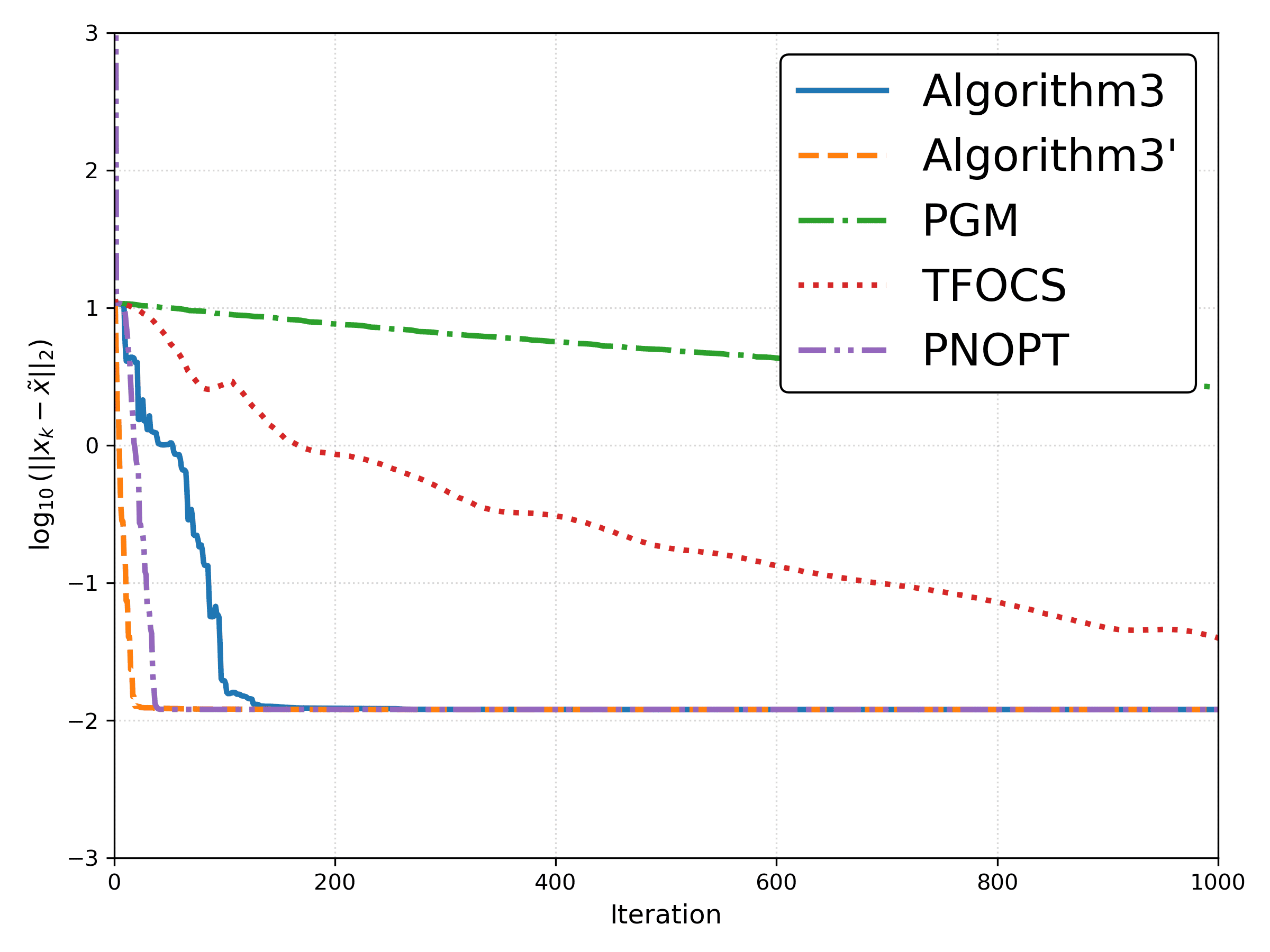}
        \caption{Convergence vs. iteration count}
    \end{subfigure}
    \caption{Convergence behavior for $(m,n,s)=(3000,500,180)$ with $\lambda = 0.1$}
    \label{fig:lasso_lamb01_convergence_3000}
\end{figure}
For both instances, Algorithm~3' demonstrates superior efficiency in CPU time, confirming the ideal balance between iteration count and per-iteration cost discussed above.
Notably, for the larger problem $(m,n,s)=(3000,500,180)$, Algorithm~3' outperforms all other methods in both metrics, exhibiting stable and rapid convergence from the early stages.

Finally, Table~\ref{table:switch_count_lasso} shows the number of times that the condition~\eqref{conditional_branch} of Algorithm~\ref{alg:proposed} and~3' were not satisfied, as well as the ratio to the total number of iterations.
\begin{table}[htbp]
\centering
\caption{Switch count and ratio for Algorithm~\ref{alg:proposed} and Algorithm 3'}
\label{table:switch_count_lasso}
\begin{tabular}{cccccccc}
\toprule
 &  & \multicolumn{3}{c}{Algorithm~\ref{alg:proposed}} & \multicolumn{3}{c}{Algorithm 3'} \\
\cmidrule(lr){3-5} \cmidrule(lr){6-8}
$\lambda$ & $(m,n,s)$ & Iter. & Switch & Ratio (\%) & Iter. & Switch & Ratio (\%) \\
\midrule
0.1  & (500, 550, 50)    & 1198.0 & 119.7 & 10.0 & 890.2  & 116.8 & 13.1 \\
0.1  & (1000, 1050, 50)  & 1746.1 & 123.1 & 7.0  & 882.2  & 122.9 & 13.9 \\
0.1  & (500, 150, 50)    & 333.8  & 39.6  & 11.9 & 232.4  & 33.7  & 14.5 \\
0.1  & (1000, 300, 60)   & 414.8  & 54.6  & 13.2 & 287.5  & 44.1  & 15.3 \\
0.1  & (3000, 500, 180)  & 296.2  & 3.7   & 1.2  & 177.3  & 5.9   & 3.3 \\
0.1  & (5000, 1500, 300) & 494.2  & 7.2   & 1.5  & 274.5  & 8.3   & 3.0 \\
0.1  & (7000, 2000, 400) & 298.4  & 0.3   & 0.1  & 175.2  & 1.0   & 0.6 \\
\midrule
0.01 & (500, 550, 50)    & 2375.9 & 1.3   & 0.1  & 1652.4 & 2.8   & 0.2 \\
0.01 & (1000, 1050, 50)  & 3310.6 & 0.1   & 0.0  & 2281.6 & 0     & 0.0 \\
0.01 & (500, 150, 50)    & 243.0  & 3.1   & 1.3  & 140.3  & 2.2   & 1.6 \\
0.01 & (1000, 300, 60)   & 303.8  & 0.2   & 0.1  & 163.7  & 0     & 0.0 \\
0.01 & (3000, 500, 180)  & 175.5  & 0.0   & 0.0  & 54.6   & 0     & 0.0 \\
0.01 & (5000, 1500, 300) & 363.2  & 0.0   & 0.0  & 408.0  & 0     & 0.0 \\
0.01 & (7000, 2000, 400) & 201.0  & 0.0   & 0.0  & 105.7  & 0     & 0.0 \\
\bottomrule
\end{tabular}
\end{table}
From the results in Table~\ref{table:switch_count_lasso}, several important trends can be observed.
First, the ratio of switches is at most 15.3\% under all settings.
Second, we observe the influence of the regularization parameter $\lambda$.
Comparing the cases $\lambda = 0.1$ and $\lambda = 0.01$, the ratio of switches generally decreases as $\lambda$ becomes smaller.
Finally, the ratio of switches tends to decrease as the number of data points $m$ increases.
\subsection{\texorpdfstring{$\ell_1$}{L1}-regularized logistic regression problems}\label{sec:ne_logistic_l1}
In this section, we evaluate the numerical performance of \( \ell_1 \)-regularized logistic regression problems, which are given as follows:
\begin{equation*}
\min_{{x} \in \mathbb{R}^n} \sum_{i=1}^m \log\left(1 + \exp(-b_i {x}^\top {a}_i)\right) + \lambda \|{x}\|_1.
\end{equation*}
where, \( m \) denotes the number of data samples, \( n \) the number of features, \( {a}_i \in \mathbb{R}^n \) the \( i \)-th training data, \( b_i \) the \( i \)-th correct label of 1 or -1, and \( \lambda \) the regularization parameter.
In these experiments, we used the binary classification datasets a1a, a9a, leukemia, and gisette scale from LIBSVM~\cite{libsvm_data}.
The details of the datasets are summarized in Table~\ref{table:logistic_dataset}. 
\begin{table}[htbp]
    \centering
    \caption{Dataset information}
    \label{table:logistic_dataset}
    \begin{tabular}{ccccc} 
        \toprule
        Dataset & a1a & a9a & leukemia & gisette scale \\ 
        \midrule
        $m$ & 1605 & 32561 & 38 & 6000 \\
        $n$ & 123 & 123 & 7129 & 5000 \\
        \bottomrule
    \end{tabular}
\end{table}
The parameters and comparison methods used in the experiments are identical to those described in Section~\ref{sec:ne_lasso}.
Algorithm~\ref{alg:proposed}, Algorithm~3' and PNOPT terminate when $\frac{\|{x}_{k}^+ - {x}_k\|}{\textnormal{max}\{1,\|{x}_k\|\}} \leq 10^{-6}$.
PGM and TFOCS terminate when $\frac{\|{x}_{k+1} - {x}_k\|}{\textnormal{max}\{1,\|{x}_k\|\}} \leq 10^{-6}$.
For the gisette scale dataset, we increase $\mu_k$ at each iteration by setting $\mu_k = \mu_{k-1}/0.9$ for both Algorithm~\ref{alg:proposed} and~3' to improve numerical stability.
Table~\ref{table:computation_time_logistic} presents the computation times, and Table~\ref{table:iteration_count_logistic} shows the number of iterations.
\begin{table}[htbp]
\centering
\caption{CPU time (s)}
\label{table:computation_time_logistic}
\begin{tabular}{ccccccc}
\toprule
$\lambda$ & Dataset & Algorithm~\ref{alg:proposed} & Algorithm 3' & PGM & TFOCS & PNOPT \\
\midrule
0.1  & a1a           & 0.6   & \textbf{0.4}   & 11.6   & 7.4   & 29.3 \\
0.1  & a9a           & 20.4  & \textbf{5.0}   & 76.1   & 78.4  & 53.6 \\
0.1  & leukemia      & \textbf{1.0}   & 1.4   & 1.1    & 1.3   & 109.7 \\
0.1  & gisette scale & 721.9 & 626.5 & 2476.3 & 883.4 & \textbf{245.9} \\
\midrule
0.01 & a1a           & 0.6   & \textbf{0.3}   & 16.6   & 11.5  & 85.9 \\
0.01 & a9a           & 15.7  & \textbf{3.2}   & 93.4   & 76.5  & 98.8 \\
0.01 & leukemia      & 0.92  & \textbf{0.64}  & 2.5    & 1.6   & 171.3 \\
0.01 & gisette scale & 855.9 & 1486.4 & 3433.6 & 895.7 & \textbf{290.3} \\
\bottomrule
\end{tabular}
\end{table}
\begin{table}[htbp]
\centering
\caption{Number of iterations}
\label{table:iteration_count_logistic}
\begin{tabular}{ccccccc}
\toprule
$\lambda$ & Dataset & Algorithm~\ref{alg:proposed} & Algorithm 3' & PGM & TFOCS & PNOPT \\
\midrule
0.1  & a1a           & 551  & 642  & 22787 & 10218 & \textbf{184} \\
0.1  & a9a           & 956  & 933  & 30998 & 13413 & \textbf{377} \\
0.1  & leukemia      & 1130 & 1534 & 1044  & \textbf{958}  & 992 \\
0.1  & gisette scale & 5482 & 6995 & 60256 & 10989 & \textbf{566} \\
\midrule
0.01 & a1a           & 494  & 447  & 32402 & 15675 & \textbf{418} \\
0.01 & a9a           & 715  & 555  & 37067 & 13485 & \textbf{513} \\
0.01 & leukemia      & 1052 & \textbf{598} & 2831 & 1178 & 2246 \\
0.01 & gisette scale & 6147 & 14319 & 82917 & 11826 & \textbf{787} \\
\bottomrule
\end{tabular}
\end{table}
Similar to the LASSO experiments, Algorithm~3' achieves the best CPU time across most datasets, significantly outperforming other first-order methods such as PGM and TFOCS.
Regarding the comparison with PNOPT, although PNOPT yields the smallest iteration counts for several datasets, its high computational overhead per iteration results in longer CPU times for the other cases. 
This confirms that Algorithm~3' strikes a superior balance between convergence speed and per-iteration cost.
Furthermore, Algorithm 3' consistently outperforms Algorithm~\ref{alg:proposed}, confirming the effectiveness of the quadratic interpolation.
Finally, Table~\ref{table:switch_count_logistic} shows the number of times that~\eqref{conditional_branch} of Algorithm~\ref{alg:proposed} and~3' were not satisfied, as well as the ratio to the total number of iterations.
\begin{table}[htbp]
\centering
\caption{Switch count and ratio for Algorithm~\ref{alg:proposed} and Algorithm 3'}
\label{table:switch_count_logistic}
\begin{tabular}{cccccccc}
\toprule
 &  & \multicolumn{3}{c}{Algorithm~\ref{alg:proposed}} & \multicolumn{3}{c}{Algorithm 3'} \\
\cmidrule(lr){3-5} \cmidrule(lr){6-8}
$\lambda$ & Dataset & Iter. & Switch & Ratio (\%) & Iter. & Switch & Ratio (\%) \\
\midrule
0.1  & a1a           & 551  & 0    & 0.0  & 642  & 1    & 0.2 \\
0.1  & a9a           & 956  & 0    & 0.0  & 933  & 0    & 0.0 \\
0.1  & leukemia      & 1130 & 119  & 10.5 & 1534 & 81   & 5.3 \\
0.1  & gisette scale & 5482 & 1551 & 28.3 & 6995 & 1833 & 26.2 \\
\midrule
0.01 & a1a           & 494  & 0    & 0.0  & 447  & 0    & 0.0 \\
0.01 & a9a           & 715  & 0    & 0.0  & 555  & 0    & 0.0 \\
0.01 & leukemia      & 1052 & 20   & 2.0  & 598  & 19   & 3.2 \\
0.01 & gisette scale & 6147 & 1585 & 25.9 & 14319 & 2258 & 15.8 \\
\bottomrule
\end{tabular}
\end{table}
For the datasets a1a and a9a the conjugate gradient direction is adopted at every iteration. 
While the switching ratio increases as $n$ increases, it remains limited to 28.3\%.
These results confirm that the proposed algorithm effectively utilizes the conjugate gradient direction in most cases.
\subsection{\texorpdfstring{$\ell_1$}{L1}-Regularized Student's \texorpdfstring{$t$}{t}-Regression}\label{sec:ne_student_t_l1}
In this section, we evaluate the numerical performance of the proposed algorithm on the $\ell_1$-regularized Student's $t$-regression \cite{Aravkin2012} problem, given by:
\begin{equation}\label{student_t_regression}
    \min_{x \in \mathbb{R}^n} \sum_{i=1}^m \log\left( 1 + \frac{(Ax-b)_i^2}{\nu} \right) + \lambda \|x\|_1.
\end{equation}
with $A \in \mathbb{R}^{m \times n}$ and $b \in \mathbb{R}^m$.
Here, $m$ is the number of data samples, $n$ is the number of features, $\lambda > 0$ is the regularization parameter, and $\nu > 0$ is a tuning parameter.
Problem \eqref{student_t_regression} is a variation of the LASSO problem given in \eqref{lasso}, where the squared loss is replaced by the Student's $t$ loss. 
This formulation is commonly employed in robust regression due to its ability to reduce the influence of outliers. 
We note that the loss function is given by $g(x)=\sum_{i=1}^{m}\log\left(1+\frac{(a_i^\top x-b_i)^2}{\nu}\right)$ is generally nonconvex.

Following the procedure in~\cite{Milzarek2014}, we generate a reference signal $x^{\mathrm{true}} \in \mathbb{R}^n$ with $k = \lfloor n/40 \rfloor$ non-zero elements at random indices. 
Each non-zero component is defined as $x^{\mathrm{true}}_i = \eta_1(i) 10^{d\,\eta_2(i)/20}$, where $\eta_1(i) \in \{-1, +1\}$ is a random sign and $\eta_2(i)$ is uniformly distributed on the interval $[0, 1]$, providing a dynamic range of $d$ dB.
The measurement matrix $A \in \mathbb{R}^{m \times n}$ with $m = n/8$ is constructed via a random discrete cosine transform (dct) such that $Ax^{\mathrm{true}} = (\text{dct}(x^{\mathrm{true}}))_J$, where $J \subset \{1, \dots, n\}$ is a random index set with $|J| = m$.
Finally, the observation vector $b$ is obtained by adding Student's $t$-distributed noise with 5 degrees of freedom to $Ax^{\mathrm{true}}$ and scaling the result by 0.1.

We set the parameters to $\lambda = 0.01$ and $\nu = 0.001$.
The signal length $n$ and dynamic range $d$ are varied as $n \in \{128, 256, 512, 1024\}$ and $d \in \{20, 40\}$, respectively.
The comparison methods, parameters, and termination criteria are consistent with those in Section~\ref{sec:ne_lasso}. 
For all methods, the initial point $x_0$ is generated from a uniform distribution over $[-10, 10]$. 
As Algorithm~\ref{alg:proposed} and Algorithm~3' exhibit nearly identical performance, we report only the results for Algorithm~\ref{alg:proposed}.

First, Table~\ref{tab:n128_d20} shows the average results over 10 problem instances for $n=128$ and $d=20$.
\begin{table}[htbp]
\centering
\caption{Numerical results for $n=128$, $d=20$}
\label{tab:n128_d20}
\begin{tabular}{lcccc}
\hline
Method & Time (s) & Iter. & Function value & Result \\
\hline
Algorithm~\ref{alg:proposed} & 5.8 & 19254.7 & 0.1618 & Converged \\
PGM      & 756.5 & 2010122.1 & 0.1618 & Converged \\
TFOCS    & 23.0  & 54076.4  & 0.3014 & Not converged \\
PNOPT    & 7.8 & 196.0 & 5.6008 & Not converged \\
\hline
\end{tabular}
\end{table}
Algorithm~\ref{alg:proposed} successfully converged to the optimal solution and exhibited superior computational efficiency compared to the other methods.
In contrast, the objective values obtained by TFOCS and PNOPT remained significantly higher than those of the other methods, suggesting that these solvers did not converge.
Although PGM eventually converged, its convergence was considerably slower than that of Algorithm~\ref{alg:proposed}.
For these reasons, we focused on Algorithm~\ref{alg:proposed} for further evaluation using different problem sizes ($n=128, 256, 512, 1024$ and $d=20, 40$). 
The results are summarized in Table~\ref{tab:compare_student_t}.
\begin{table}[htbp]
\caption{Numerical performance for different problem settings}\label{tab:compare_student_t}
\begin{tabular*}{\textwidth}{@{\extracolsep{\fill}}cccccc}
\toprule
& & \multicolumn{2}{c}{$d = 20$} & \multicolumn{2}{c}{$d = 40$} \\
\cmidrule(lr){3-4} \cmidrule(lr){5-6}
$n$ & Method & Time (s) & Iter. & Time (s) & Iter. \\
\midrule
128  & Algorithm~\ref{alg:proposed} & 5.8  & 18861.6 & 14.4 & 36030.5 \\
256  & Algorithm~\ref{alg:proposed} & 21.5 & 46029.0 & 25.2 & 54493.8 \\
512  & Algorithm~\ref{alg:proposed} & 26.4 & 62904.0 & 37.5 & 71611.5 \\
1024 & Algorithm~\ref{alg:proposed} & 52.1 & 101578.3 & 34.8 & 86031.6 \\
\bottomrule
\end{tabular*}
\end{table}
Despite this being a nonconvex optimization problem, it was confirmed that Algorithm~\ref{alg:proposed} successfully converged in all cases.
Finally, Table~\ref{table:switch_count_student_t} shows the number of times that the condition~\eqref{conditional_branch} of Algorithm~\ref{alg:proposed} was not satisfied, as well as the ratio to the total number of iterations.
\begin{table}[htbp]
\centering
\caption{Switch count and ratio for Algorithm~\ref{alg:proposed}}
\label{table:switch_count_student_t}
\begin{tabular}{ccccc}
\toprule
$n$ & $d$ & Iter. & Switch & Ratio (\%) \\
\midrule
128  & 20 & 18861.6  & 1934.5  & 10.3 \\
128  & 40 & 36030.5  & 3157.1  & 8.8  \\
256  & 20 & 46029.0  & 7664.2  & 16.7 \\
256  & 40 & 54493.8  & 3062.7  & 5.6  \\
512  & 20 & 62904.0  & 9381.5  & 14.9 \\
512  & 40 & 71611.5  & 2855.5  & 4.0  \\
1024 & 20 & 101578.3 & 16883.2 & 16.6 \\
1024 & 40 & 86031.6  & 3638.8  & 4.2  \\
\bottomrule
\end{tabular}
\end{table}
We observe that the switching ratio is only 16.7\%, regardless of the nonconvexity of $g$.

\subsection{Least Squares Problem with MCP}\label{sec:ne_ls_MCP}
In this section, we evaluate the numerical performance of the Algorithm~\ref{alg:proposed_weakly} for the following least squares problem with the MCP (Minimax Concave Penalty):
\begin{equation*}
    \min_{x \in \mathbb{R}^n} \|Ax-b\|^2 + \sum_{j=1}^n p_{\lambda, c}(x_j).
\end{equation*}
Here, $A \in \mathbb{R}^{m \times n}$ and $b \in \mathbb{R}^m$ are defined in the same way as in Section~\ref{sec:ne_lasso}.
The function $p_{\lambda, c}(t)$ is the MCP function with parameters $\lambda > 0$ and $c > 0$, defined as follows:
\begin{equation*}
    p_{\lambda, c}(x_j) =
    \begin{cases}
        \lambda |x_j| - \frac{x_j^2}{2c}, & (|x_j| \le c \lambda), \\
        \frac{1}{2}c \lambda^2, & (|x_j| > c \lambda).
    \end{cases}
\end{equation*}
Although the MCP function $p_{\lambda, c}(x_j)$ is nonconvex, it is known to be $\rho$-weakly convex with $\rho = 1/c$.
In all experiments, $c$ is set to 0.1 and 10, where the smaller value $c=0.1$ yields a stronger nonconvexity.
Since the weak convexity of the regularization term does not generally guarantee uniqueness of the proximal mapping, we apply the parameter selection $\mu_k \in (0,1/\rho)$ from Section~\ref{sec:weakly_convex} to ensure strong convexity of the subproblems for PGM and TFOCS.
PNOPT is excluded from the comparison as its convergence is not theoretically guaranteed for the nonconvex problems considered here.
Datasets and termination criteria follow Section~\ref{sec:ne_lasso}.
In addition to Algorithm~\ref{alg:proposed_weakly}, we evaluate Algorithm~4', a variant incorporating the quadratic interpolation strategy described in Algorithm~3'.

Tables~\ref{table:computation_time_MSE_MCP} and~\ref{table:iteration_count_MSE_MCP} show the computation time and the iteration count, respectively.
\begin{table}[htbp]
\centering
\caption{CPU time (s)}
\label{table:computation_time_MSE_MCP}
\begin{tabular}{c c c c c c}
\toprule
($\lambda$,$\theta$) & $(m,n,s)$ & Algorithm~\ref{alg:proposed_weakly} & Algorithm 4' & PGM & TFOCS \\
\midrule
(0.1,0.1) & (500, 550, 50)   & 181.5 & \textbf{64.2}  & -      & - \\
(0.1,0.1) & (1000, 1050, 50) & 842.3 & \textbf{214.6} & -      & - \\
(0.1,0.1) & (500, 150, 50)   & 0.3   & \textbf{0.1}   & 1.8    & - \\
(0.1,0.1) & (1000, 300, 60)  & 2.6   & \textbf{0.6}   & 8.5    & - \\
(0.1,0.1) & (3000, 500, 180) & 17.0  & \textbf{2.1}   & 52.6   & - \\
(0.1,0.1) & (5000, 1500, 300)& 248.7 & \textbf{19.2}  & 1100.6 & - \\
(0.1,0.1) & (7000, 2000, 400)& 203.2 & \textbf{25.3}  & 445.9  & - \\
\midrule
(0.1,10)  & (500, 550, 50)   & 5.9   & \textbf{1.4}   & 51.9   & - \\
(0.1,10)  & (1000, 1050, 50) & 73.3  & \textbf{9.3}   & -      & - \\
(0.1,10)  & (500, 150, 50)   & 0.3   & \textbf{0.1}   & 1.5    & - \\
(0.1,10)  & (1000, 300, 60)  & \textbf{2.1}   & \textbf{2.1}   & 7.1    & - \\
(0.1,10)  & (3000, 500, 180) & 13.2  & \textbf{2.2}   & 49.3   & - \\
(0.1,10)  & (5000, 1500, 300)& 207.9 & \textbf{47.7}  & 1123.5 & 419.4 \\
(0.1,10)  & (7000, 2000, 400)& 184.5 & \textbf{25.9}  & 438.7  & 279.4 \\
\bottomrule
\end{tabular}
{\footnotesize A dash (–) indicates that the algorithm did not converge within 50,000 iterations.}
\end{table}
\begin{table}[htbp]
\centering
\caption{Number of iterations}
\label{table:iteration_count_MSE_MCP}
\begin{tabular}{c c c c c c}
\toprule
($\lambda$,$\theta$) & $(m,n,s)$ & Algorithm~\ref{alg:proposed_weakly} & Algorithm 4' & PGM & TFOCS \\
\midrule
(0.1,0.1) & (500, 550, 50)   & 23490.0 & \textbf{23116.4} & -        & - \\
(0.1,0.1) & (1000, 1050, 50) & 20025.3 & \textbf{18805.6} & -        & - \\
(0.1,0.1) & (500, 150, 50)   & 309.3   & \textbf{213.7}   & 3484.5   & - \\
(0.1,0.1) & (1000, 300, 60)  & 432.6   & \textbf{330.4}   & 7002.7   & - \\
(0.1,0.1) & (3000, 500, 180) & 312.3   & \textbf{183.2}   & 10159.0  & - \\
(0.1,0.1) & (5000, 1500, 300)& 529.0   & \textbf{251.0}   & 40442.6  & - \\
(0.1,0.1) & (7000, 2000, 400)& 299.4   & \textbf{165.9}   & 8988.0   & - \\
\midrule
(0.1,10)  & (500, 550, 50)   & 1005.2 & \textbf{636.9}  & 45890.6 & - \\
(0.1,10)  & (1000, 1050, 50) & 1643.9 & \textbf{951.2}  & -       & - \\
(0.1,10)  & (500, 150, 50)   & 269.2  & \textbf{178.4}  & 3031.2  & - \\
(0.1,10)  & (1000, 300, 60)  & 410.5  & \textbf{288.9}  & 6255.4  & - \\
(0.1,10)  & (3000, 500, 180) & 323.6  & \textbf{176.7}  & 9803.0  & - \\
(0.1,10)  & (5000, 1500, 300)& \textbf{511.9}  & \textbf{511.9}  & 38999.2 & 9912.4 \\
(0.1,10)  & (7000, 2000, 400)& 306.1  & \textbf{185.2}  & 8701.5  & 3944.2 \\
\bottomrule
\end{tabular}
{\footnotesize A dash (–) indicates that the algorithm did not converge within 50,000 iterations.}
\end{table}
From these results, we observe that Algorithm~4' consistently outperforms the comparison methods across most datasets for both $c=0.1$ and $c=10$.
Notably, TFOCS fails to converge within the 50,000 iteration limit for several problems.
In contrast, Algorithm~\ref{alg:proposed_weakly} and~4' successfully converge in all cases, maintaining stable performance regardless of problem size.
Furthermore, Algorithm~4' outperforms Algorithm~\ref{alg:proposed_weakly} in all cases, demonstrating the effectiveness of the quadratic interpolation-based line search.
Finally, Table~\ref{table:switch_count_ls_MCP} shows the number of times that the condition~\eqref{conditional_branch} was not satisfied, as well as the ratio to the total number of iterations.
\begin{table}[htbp]
\centering
\caption{Switch count and ratio for Algorithm~\ref{alg:proposed_weakly} and Algorithm 4'}
\label{table:switch_count_ls_MCP}
\begin{tabular}{cccccccc}
\toprule
 &  & \multicolumn{3}{c}{Algorithm~\ref{alg:proposed_weakly}} & \multicolumn{3}{c}{Algorithm 4'} \\
\cmidrule(lr){3-5} \cmidrule(lr){6-8}
$(\lambda,\theta)$ & $(m,n,s)$ & Iter. & Switch & Ratio (\%) & Iter. & Switch & Ratio (\%) \\
\midrule
(0.1,0.1) & (500, 550, 50)    & 23490.0 & 695.7 & 3.0 & 23116.4 & 800.5 & 3.5 \\
(0.1,0.1) & (1000, 1050, 50)  & 20025.3 & 226.9 & 1.1 & 18805.6 & 304.2 & 1.6 \\
(0.1,0.1) & (500, 150, 50)    & 309.3   & 31.6  & 10.2 & 213.7   & 24.4  & 11.4 \\
(0.1,0.1) & (1000, 300, 60)   & 432.6   & 55.9  & 12.9 & 330.4   & 49.3  & 14.9 \\
(0.1,0.1) & (3000, 500, 180)  & 312.3   & 4.9   & 1.6 & 183.2   & 6.6   & 3.6 \\
(0.1,0.1) & (5000, 1500, 300) & 529.0   & 8.1   & 1.5 & 251.0   & 8.8   & 3.5 \\
(0.1,0.1) & (7000, 2000, 400) & 299.4   & 1.0   & 0.3 & 165.9   & 1.5   & 0.9 \\
\midrule
(0.1,10)  & (500, 550, 50)    & 1005.2 & 97.8 & 9.7 & 636.9 & 73.8 & 11.6 \\
(0.1,10)  & (1000, 1050, 50)  & 1643.9 & 108.5 & 6.6 & 951.2 & 104.9 & 11.0 \\
(0.1,10)  & (500, 150, 50)    & 269.2  & 26.0 & 9.7 & 178.4 & 19.7 & 11.0 \\
(0.1,10)  & (1000, 300, 60)   & 410.5  & 50.7 & 12.4 & 288.9 & 48.4 & 16.8 \\
(0.1,10)  & (3000, 500, 180)  & 323.6  & 5.7 & 1.8 & 176.7 & 6.5 & 3.7 \\
(0.1,10)  & (5000, 1500, 300) & 511.9  & 6.9 & 1.4 & 547.7 & 11.4 & 2.1 \\
(0.1,10)  & (7000, 2000, 400) & 306.1  & 1.1 & 0.4 & 185.2 & 1.5 & 0.8 \\
\bottomrule
\end{tabular}
\end{table}
It can be observed that the ratio of switches is at most 16.8\%, indicating that the number of switches is sufficiently small.

\section{Concluding remarks}\label{sec:5}
In this paper, we proposed the proximal nonlinear conjugate gradient method for solving composite minimization problems.
The method defines the gradient via the \textit{forward-backward residual} and employs a search direction based on the three-term HS method.
Furthermore, when the objective function is strongly convex and the parameters are appropriately chosen, the proposed method coincides with nonlinear conjugate gradient methods, showing that it is a natural extension of these methods.
We have established the global convergence of the algorithm under standard assumptions even for weakly convex nonsmooth functions and analyzed its convergence rate in terms of stationarity. 
In addition, when the smooth term is strongly convex, we have proved the convergence of the generated sequence to the optimal solution along with its convergence rate.
Numerical experiments demonstrate that the proposed method exhibits stable performance and consistently outperforms TFOCS and PNOPT in both convex and nonconvex settings.
This paper is the first attempt to apply the nonlinear conjugate gradient methods to minimize composite function.
Based on the development of nonlinear conjugate gradient methods, further improvements and new method proposals are expected in the future. 
Future work includes to develop a method that guarantees global convergence without switching to the proximal gradient methods.
\backmatter
\bmhead{Acknowledgements}
This work was supported by Japan Society for the Promotion of Science (JSPS) KAKENHI Grant Numbers JP26K14720 and J23K10999.
The authors would like to thank Dr. Shummin Nakayama for helpful comments on the numerical experiments.
\section*{Declarations}
\bmhead{Funding}  
This work was supported by Japan Society for the Promotion of Science (JSPS) KAKENHI Grant Numbers JP26K14720 and JP23K10999.
\bmhead{Competing interests}  
The authors have no competing interests to declare that are relevant to the content of this article.
\bmhead{Data Availability}
Data sets generated during the current study are available from the corresponding author on reasonable request.

\bibliography{main}

\end{document}